\documentclass[svgnames]{article}
\usepackage{geometry}
\usepackage[utf8]{inputenc}
\usepackage{fnpct} % for multiple footnotes
\usepackage{amsmath}
\usepackage{amsfonts}
\usepackage{amssymb} 
\usepackage[cal=euler,calscaled=.98]{mathalpha}
\usepackage{mathtools} % for \prescript{top}{bottom}{argument}
\usepackage{tikz} % for tikz
\usepackage{tikz-cd} % for \tikzcd
\usepackage{amsthm} % needed for theoremstyle and for proofs
\usepackage{adjustbox} % for rescaling
\usepackage{microtype}
\usepackage[%
colorlinks=true%
,linkcolor=FireBrick%
,urlcolor=SlateBlue%
,citecolor=Teal%
]{hyperref}

\numberwithin{equation}{section}

\theoremstyle{definition} % thm, prop, cor, lemma and examples NOT in italic
\newtheorem{thm}{Theorem}[section] 
\newtheorem{lemma}[thm]{Lemma}
\newtheorem{prop}[thm]{Proposition}
\newtheorem{cor}[thm]{Corollary}
\newtheorem{ex}[thm]{Example}
\newtheorem{defn}[thm]{Definition}
\newtheorem{rmk}[thm]{Remark}
\newtheorem{note}[thm]{Notation}

\newtheorem*{thm*}{Theorem} % Unnumbered theorem (useful for introduction)

\newcommand{\mathsc}[1]{{\normalfont\textsc{#1}}}
\DeclareMathOperator{\isom}{isom}
\DeclareMathOperator{\commutassoc}{coas} %% symbol for commutivity composed with associativity, i.e. (a+b)+(c+d) -> (a+c)+(b+d)
\DeclareMathOperator{\Aut}{Aut}\newcommand{\autom}{\Aut}
\DeclareMathOperator{\Hom}{Hom}
\DeclareMathOperator{\objects}{ob}
\DeclareMathOperator{\opposite}{op}
\DeclareMathOperator{\identity}{id}\newcommand{\id}{\identity}
\DeclareMathOperator{\barfunc}{bar}
\DeclareMathOperator{\dfunc}{d}
\newcommand{\Dbar}{\delta_{D}}
\DeclareMathOperator{\Det}{\mathsc{Det}}

\DeclareMathOperator{\Grpd}{\mathsc{Grpd}}
\DeclareMathOperator{\Pic}{\mathsc{Pic}}
\DeclareMathOperator{\TrCat}{\mathsc{TrCat}}

\newcommand{\num}[1]{\mathbb{#1}}
\newcommand{\ZZ}{\num{Z}}

\newcommand{\leftmapsto}{\mathrel{\reflectbox{\ensuremath{\mapsto}}}}

\newcommand{\picadd}{+} %% Addition in Picard Groupoids
\newcommand{\picmult}{\odot} %% Multiplication in V(T)
\newcommand{\action}{\cdot} %% Action in \pi_1
\DeclarePairedDelimiter\GenericDetInput{\lbrack}{\rbrack} %% Uses mathtools
\newcommand{\GenericDet}{\GenericDetInput{-}} %% Generic det

\begin{document}

\title{On Multi-Determinant Functors for Triangulated Categories}
\author{%
  Ettore Aldrovandi \\
  {\small \url{aldrovandi@math.fsu.edu}}
  \and 
  Cynthia Lester\footnote{Current address: The College of Wooster 1189 Beall Ave, Wooster, OH 44691} \\
  {\small \url{clester@wooster.edu}}
  \and {\small Department of Mathematics, Florida State University, Tallahassee, FL 32306-4510}}
\date{}
\maketitle

\begin{abstract}
We extend Deligne's notion of determinant functor to tensor triangulated categories.
Specifically, to account for the multiexact structure of the tensor, we define a determinant functor on the 2-multicategory of triangulated categories and we provide a multicategorical version of the universal determinant functor for triangulated categories whose multiexactness properties are conveniently captured by a certain complex modeled by cubical shapes, which we introduce along the way. We then show that for a tensor triangulated category whose tensor admits a Verdier structure the resulting determinant functor takes values in a categorical ring.
\end{abstract}

\setcounter{tocdepth}{2}
\tableofcontents

\section{Introduction}

In \cite{deligne1987determinant}, Deligne defined a determinant functor $\det\colon \mathcal{E}\to\mathcal{P}$ on an exact category $\mathcal{E}$ to a Picard groupoid $\mathcal{P}$.
This functor, defined only on the isomorphism classes of $\mathcal{E}$, $\det\colon\text{isom}(\mathcal{E})\to\mathcal{P}$ is subject to the constraint that every short exact sequence $X\to Y\to Z$ in $\mathcal{E}$ corresponds a natural isomorphism $\det(Z)\otimes\det(X)\cong \det(Y)$ in $\mathcal{P}$.  Furthermore, Deligne constructed a Picard groupoid $V(\mathcal{E})$, whose objects are called \emph{virtual objects,} and a \emph{universal} determinant functor $\det\colon \mathcal{E}\to V(\mathcal{E})$ such that $\pi_0(V(\mathcal{E}))$ is $K_0(\mathcal{E})$ and $\pi_1(V(\mathcal{E}))$ is $K_1(\mathcal{E})$, where the $K_i$ are the $K$-theory groups of $\mathcal{E}$. To say $\det\colon \mathcal{E} \to V(\mathcal{E})$ is universal means that every determinant functor $\det'\colon \mathcal{E} \to \mathcal{P}$ factors uniquely, in a 2-categorical sense, through $\det\colon \mathcal{E} \to V(\mathcal{E})$ by way of a morphism $V(\mathcal{E}) \to \mathcal{P}$ of Picard groupoids.

In \cite{breuning2011determinant}, Breuning extended Deligne's definition of determinant functors, virtual objects, and universal determinant functors to (small) triangulated categories. Since triangulated categories do not have an immediately defined $K$-theory spectrum (say, in the same way as exact or Waldhausen categories do), the entirety of Deligne's results cannot be extended to triangulated categories in a straightforward way.  However, we can still discuss the universal determinant functor $\det\colon \mathcal{T} \to V(\mathcal{T})$ for a triangulated category $\mathcal{T}$, and compute $\pi_0(V(\mathcal{T}))$ and $\pi_1(V(\mathcal{T}))$.  Moreover, $\pi_0(V(\mathcal{T}))$ is the Grothendieck group for $\mathcal{T}$, i.e.\ it is what $K_0(\mathcal{T})$ should be in any reasonable definition of $K$-theory for $\mathcal{T}$. Therefore, one definition for $K_1(\mathcal{T})$ is $\pi_1(V(\mathcal{T}))$.  Breuning's results in \cite{breuning2011determinant} show that this leads to an interesting and useful theory. As remarked there, there is a similarity with the simplicial ideas in Neeman's constructions (see \cite{neeman2005K-triangulated}). Breuning's techniques were used by Muro and Tonks \cite{muro2007type}, and, later, by Muro, Tonks and Witte \cite{muro2010determinant} to obtain explicit models for the universal Picard groupoid $V(\mathcal{C})$, where $\mathcal{C}$ is either a Waldhausen or a triangulated category.

This paper grew out of our consideration of tensor triangulated categories, namely triangulated categories equipped with a (symmetric) monoidal structure compatible with the triangulation \cite{balmer2010ttg}. The compatibility dictates that the bifunctor underlying the tensor structure be exact, i.e.\ that it preserve the distinguished triangles, in each of the two variables while the other is kept fixed (plus some additional conditions, which we shall ignore, for the time being, for the sake of simplicity). 

If $\mathcal{T}$ is a tensor triangulated category, we ask what additional properties the Picard groupoid $V(\mathcal{T})$ and the universal determinant functor $\det \colon \mathrm{isom}(\mathcal{T}) \to V(\mathcal{T})$  must have so that they are compatible with the tensor triangulated structure. (A similar question was considered for Waldhausen categories by Muro and Tonks \cite{muro2007type}.) We show that if $\mathcal{T}$ is a tensor triangulated category, then $V(\mathcal{T})$ can be made into a categorical ring. In fact, it turns out categorical rings are the natural recipients for determinant functors on tensor triangulated categories.

This kind of question is most adequately discussed using (colored) operads, or in other parlance, multicategories. This is because ring-like structures, such as those carried by a tensor triangulated category or a categorical ring, can be viewed as multiexact functors, that is, $n$-to-one functors which are exact, in the appropriate way, in each variable. More specifically, multiexact functors of triangulated categories are just $n$-functors
\begin{equation*}
  F\colon \mathcal{T}_1\times \dots \times \mathcal{T}_n\to \mathcal{T}
\end{equation*}
which preserve distinguished triangles in each variable (again, plus a number of  compatibility conditions). Then triangulated categories form a groupoid enriched multicategory where the categories of $n$-fold morphisms are precisely given by the multiexact functors and their natural isomorphisms \cite{schnurer2018sixop}.\footnote{This is slightly more restrictive than \cite{schnurer2018sixop}, which considers enrichments over linear categories.} Similarly, Picard groupoids can be made into a groupoid enriched multicategory by adopting an analogous definition of multiexact functor (see \cite{ducrot2005cube,aldrovandi2015biex}). Ring-like objects in both of these are the monoid objects in the multicategorical sense.

Thus we extend the definition of a determinant functor to multicategories, and we call the result a \emph{multi-determinant functor.}  The definition of multi-determinant functor can be given in two ways: as a clear extension of Breuning's determinant functor, and as a functor between ``complexes of cubes.''  The former is the application of the idea of multiexactness to generalize Breuning's axioms of the determinant functor: we require that the standard axioms hold in each variable, plus certain compatibility conditions for the additivity data when two or more variables are involved. The latter are (truncated) complexes with a cubical shape modeled after Mac Lane's $Q$-construction \cite{maclane1956ann}. Like the original definition, they encode very precisely all the data expressing the multiexactness condition. (Additional inspirations for the definition of the cubical complexes come from the cubes and the $Q$-complexes used, respectively, in refs.\ \cite{burgos2000arith,mccarthy1996chain}.)

Moreover, our multi-determinant functor lends itself to an extension of the 2-categorical properties of the universal determinant functor. Specifically, the classical definition of the universal determinant functor $\det\colon \mathcal{T}\to V(\mathcal{T})$ is equivalent to saying that for any triangulated category $\mathcal{T}$ and Picard groupoid $\mathcal{P}$ there is an equivalence of groupoids
\begin{equation*}
  \Det(\mathcal{T},\mathcal{P}) \simeq \Pic(V(\mathcal{T}),\mathcal{P})\,.
\end{equation*}
In other words, for a given $\mathcal{T}$, the universal determinant co-represents the 2-functor 
\begin{equation*}
  \mathcal{P}\rightsquigarrow \Det(\mathcal{T},\mathcal{P})\,.
\end{equation*}
We extend this equivalence, with the usual functorial properties, to the multicategorical situation in section
\ref{section_universal_multideterminants}. Specifically, we show 
\begin{thm*}[Theorem \ref{main_theorem_about_multi-determinants} and Corollary \ref{categorical_equiv_of_DET_and_PIC} without technical assumptions]
  Let $\TrCat$ be the (groupoid enriched) multicategory of triangulated categories and $\Pic$ that of Picard groupoids.\footnote{$\Pic$, both as a category and as a multicategory, is enriched over itself; in this statement we consider it as enriched over groupoids.} Denote by $\Pic^\mathnormal{\times}$ the corresponding (groupoid enriched) monoidal category whose objects are strings $(\mathcal{P}_0,\dots,\mathcal{P}_n)$. There is an equivalence of groupoids
  \begin{equation*}
    \Det(\mathcal{T}_0,\dots,\mathcal{T}_n;-) \simeq \Pic^\mathnormal{\times}(V(\mathcal{T}_0),\dots,V(\mathcal{T}_n);-)\,,
  \end{equation*}
  where the left-hand side consists of multi-determinants and is considered as a 2-functor from $\Pic^\mathnormal{\times}$ to groupoids. In other words, the groupoid of multi-determinants is co-represented by the object $(\mathcal{P}_0,\dots,\mathcal{P}_n)$ of $\Pic^\mathnormal{\times}$.
\end{thm*}
It turns out that to fully account for the functoriality mentioned in the above statement it is necessary to impose some mild restrictions on the definition of $\TrCat$: that the (multi)morphisms create and preserve ``Verdier structures" on certain $3\times 3$ diagrams (see \cite{bbd1981pervers,keller2002connection,may2001ttt}, and Definitions \ref{Verdier_Structure} and
\ref{def_multifunctor_verdier}). The corresponding multicategory $\TrCat_{\mathcal{V}} \subset \TrCat$ is introduced in section \ref{sec:funct_equiv}, and we use it in section \ref{sec:det_ttt} to return to the question of the structure inherited by the Picard groupoid $V(\mathcal{T})$ when $\mathcal{T}$ is tensor triangulated. For these ``nice'' tensor triangulated categories $\mathcal{T}$, the classic virtual object $V(\mathcal{T})$ is a categorical ring (see Theorem
\ref{cat_ring_for_tensor_triangulated}), with the straightforward consequence that $K_0(\mathcal{T})$ is a ring, and $K_1(\mathcal{T})$ is a $K_0(\mathcal{T})$-bimodule.

This paper is organized as follows. We recall some background notions concerning determinant functors in Section \ref{sec:background}.  Our main development of multi-determinant functors is carried out in Section \ref{sec:multi-determinant-funct-triangulated}; the generalization of Breuning's determinant functor in Section \ref{sec:definition_multideterminant}, and the alternative one based on the cubical complex in Section \ref{sec:alternative_definition}. The existence of universal multi-determinant functors is shown in Section
\ref{section_universal_multideterminants}. The short last section (Section \ref{sec:det_ttt}) contains the application to the tensor structure and the proof that the universal determinant lands in a categorical ring.

%%%%%%%%%%%%%%%%%%%%
\section{Background}
\label{sec:background}
%%%%%%%%%%%%%%%%%%%%

We recall a few basic notions about Picard groupoids, and then review the classic definitions related to determinant functors and fix some notation. Then we discuss some specific categories; fixing some notation and vocabulary as well as stating some related miscellaneous facts.

\subsection{Picard Groupoids}
\label{sec:picard_groupoids}
%%%%%%%%%%%%%%%%%%%%%%%%%%%%%%%%

By a Picard groupoid, also known as a categorical group, \cite{MR1250465}, we mean a groupoid $\mathcal{P}$ equipped with a symmetric monoidal structure, denoted by $\picadd \colon \mathcal{P} \times \mathcal{P} \to \mathcal{P}$, which is \emph{group-like} in the sense that for every object $x$ of $\mathcal{P}$ the functor $x \picadd (\cdot) \colon \mathcal{P}\to \mathcal{P}$ is a natural equivalence. The monoidal structure is always assumed to be unital, and we denote a chosen unit by $0$. The symmetry consists of functorial isomorphisms $c_{a,b}\colon a + b \to b + a$ (this is the braided part) such that $c_{b,a}\circ c_{a,b}=\id_{a+b}$.

We shall denote a Picard groupoid by $(\mathcal{P},\picadd,0)$, or simply by $\mathcal{P}$, suppressing all the structure in the notation. Despite the shorthand, the monoidal structure and the unit are \emph{not} assumed to be strict.

For any Picard groupoid $\mathcal{P}$ we let $\pi_0(\mathcal{P})$ be the abelian group of isomorphism classes of objects, and $\pi_1(\mathcal{P}) = \Aut_{\mathcal{P}}(0)$ (also an abelian group). It is well known that  a Picard groupoid is classified by the quadratic map $\pi_0 (\mathcal{P})\to \pi_1(\mathcal{P})\otimes \ZZ/2$—the $k$-invariant—given by $a \mapsto c_{a,a}$
\cite{MR1250465,muro2010determinant}. (The symmetry implies that $c_{a,a}$ must have order 2.)  Here, following ref.\ \cite{MR2369166}—and also an older remark by Mac~Lane \cite{maclane1956ann}—we prefer to use the morphism
\begin{equation*}
  \commutassoc_{a,b,c,d} \colon (a + b) + (c + d) \longrightarrow (a + c) + (b + d)\,,
\end{equation*}
which includes both the commutativity and associativity isomorphisms. The coherence condition for $\commutassoc$ includes and is equivalent to both the pentagon and hexagon diagrams. This mirrors the fact that the stable cohomology of $\pi_0$ with coefficient in $\pi_1$—where the $k$-invariant lives—can be computed by the Eilenberg-Mac\~Lane cubical complex. The morphism $\commutassoc$ defines the corresponding class in the latter. (See section \ref{sec:multi-determinant-funct-triangulated} below.)

\subsection{Determinant Functors}
\label{sec:determinat_functors}
%%%%%%%%%%%%%%%%%%%%%%%%%%%%%%%%%

In this section, we review the basics of determinant functors. For additional information, we refer the reader to \cite{breuning2011determinant} and \cite{muro2010determinant}.

\begin{defn}\label{det_functor}
For a triangulated category $\mathcal{T}$ and a Picard groupoid $\mathcal{P}$, a \emph{determinant functor} from $\mathcal{T}$ to $\mathcal{P}$ consists of a functor
\begin{equation*}
  \GenericDet\colon \isom(\mathcal{T})\to\mathcal{P}
\end{equation*}
from the category of isomorphisms of $\mathcal{T}$ to $\mathcal{P}$, together with \emph{additivity data:} for any distinguished triangle
\begin{equation*}
  \Delta\colon x\to y\to z\to \Sigma x
\end{equation*}
there is assigned a morphism
\begin{equation*}
  \GenericDetInput{\Delta}\colon \GenericDetInput{z}+\GenericDetInput{x}\to\GenericDetInput{y}
\end{equation*}
in $\mathcal{P}$, which is natural with respect to isomorphisms of distinguished triangles. The additivity data must satisfy the following axioms:
\begin{enumerate}
\item \emph{(Commutativity)} For any $x,y\in\mathcal{T}$, and triangles
\begin{align*}
  \Delta_1 &\colon x\to x\oplus y\to y\to \Sigma x\\
  \Delta_2 &\colon y\to x\oplus y\to x\to \Sigma y.
\end{align*}
there is a commutative diagram
\begin{equation*}
  \begin{tikzcd}
    \GenericDetInput{x}+\GenericDetInput{y}
    \arrow{rr}{\cong}[swap]{\text{commutativity}} 
    \arrow{rd}[swap]{\GenericDetInput{\Delta_2}} && 
    \GenericDetInput{y}+\GenericDetInput{x} 
    \arrow[ld,"\GenericDetInput{\Delta_1}"]\\
    & \GenericDetInput{x\oplus y} &
  \end{tikzcd}
\end{equation*}
\item \emph{(Associativity)} For any octahedron:
\begin{equation*}
  \begin{tikzcd}
    &[-2em] & \Delta_{3} & \Delta_{4} \\[-3ex]
    \Delta_{1} & x \arrow[r] \arrow[d, equal] & y \arrow[r] \arrow[d] & 
                z' \arrow[d] \arrow[r] & \Sigma x \arrow[d, equal]\\
    \Delta_{2} & x \arrow[r] & z \arrow[r] \arrow[d] & 
                y' \arrow[d] \arrow[r] & \Sigma x \\
    & & x' \arrow[r, equal] \arrow[d] & x'\arrow[d] \\
    & & \Sigma y \arrow[r] & \Sigma z'
  \end{tikzcd}
\end{equation*}
the following diagram commutes
\begin{equation*}
  \begin{tikzcd}[row sep=large]
    (\GenericDetInput{x'}+\GenericDetInput{z'})+\GenericDetInput{x} \arrow{rr}{\cong}[swap]{\text{associativity}} \arrow{d}[swap]{\GenericDetInput{\Delta_4}+\id} && \GenericDetInput{x'}+(\GenericDetInput{z'}+\GenericDetInput{x}) \arrow[d,"\id+\GenericDetInput{\Delta_1}"]\\
    \GenericDetInput{y'}+\GenericDetInput{x}\arrow{r}[swap]{\GenericDetInput{\Delta_2}} &\GenericDetInput{z} & \GenericDetInput{x'}+\GenericDetInput{y}\arrow[l,"\GenericDetInput{\Delta_3}"]
  \end{tikzcd}
\end{equation*}
\end{enumerate}
\end{defn}

By abuse of language, we denote a determinant simply as $\GenericDet \colon \mathcal{T} \to \mathcal{P}$, leaving out the specification of the subcategory of isomorphisms and the additivity data.

\begin{defn}\label{universal_det_definition}
A determinant functor $\GenericDet\colon \mathcal{T}\to \mathcal{P}$ is \emph{universal} if 
for any Picard groupoid $\mathcal{P}'$ and determinant $D\colon\mathcal{T}\to\mathcal{P}'$, there exists a pair $(f,\alpha)$ comprised of a morphism of Picard groupoids $f\colon \mathcal{P}\to\mathcal{P}'$ and a natural transformation $\alpha\colon f\circ\GenericDet \Rightarrow D$ such that for any distinguished triangle $\Delta\colon x\to y\to z\to\Sigma x$ in $\mathcal{T}$ the diagram
\begin{equation*}
  \begin{tikzcd}
    f(\GenericDetInput{z})+ f(\GenericDetInput{x}) \arrow[r] \arrow{d}[swap]{\alpha_z+\alpha_x} & f(\GenericDetInput{z} + \GenericDetInput{x}) \arrow[r,"f(\GenericDetInput{\Delta})"] & f(\GenericDetInput{y}) \arrow[d,"\alpha_y"]\\
    D(z) + D(x) \arrow{rr}[swap]{D(\Delta)} && D(y)
  \end{tikzcd}
\end{equation*}
commutes. Moreover, the pair $(f,\alpha)$ is unique in the sense that if $(f',\alpha')$ is another such pair, then there exists a unique natural transformation $\beta\colon f\Rightarrow f'$ with $\alpha'\circ(\beta\ast\GenericDet) = \alpha$.
\end{defn}

\begin{note}
For any triangulated category $\mathcal{T}$, we use $\det\colon\mathcal{T}\to V(\mathcal{T})$ to represent the universal determinant functor. Additionally, we call the objects of $V(\mathcal{T})$  \emph{virtual objects.} Such universal determinant functors are known to exist for small triangulated categories (for example, see \cite{muro2010determinant}).
\end{note}

\begin{defn}\label{det_morph}
Let $\mathcal{T}$ be a triangulated category and $\mathcal{P}$ be a Picard groupoid.
For two determinant functors $d_i\colon\mathcal{T}\to\mathcal{P}$, a \emph{morphism between the determinant functors} $\theta\colon d_1\Rightarrow d_2$ is a natural transformation of the functors $d_i\colon\isom(\mathcal{T})\to\mathcal{P}$ that is compatible with the additivity data; i.e. a natural transformation such that for any distinguished triangle $\Delta\colon x\to y\to z\to\Sigma x$ the diagram below commutes.
\begin{equation*}
  \begin{tikzcd}[column sep=large]
    d_1(z)+d_1(x)\arrow[r,"d_1(\Delta)"] \arrow{d}[swap]{\theta_z+\theta_x} & 
    d_1(y) \arrow[d,"\theta_y"]\\
    d_2(z)+d_2(x) \arrow{r}[swap]{d_2(\Delta)} & d_2(y)
  \end{tikzcd}
\end{equation*}
\end{defn}

\subsection{Special Categories}
\label{sec:special_categories}

In this section we briefly mention some notation and relevant facts concerning the categories used in the next sections.  In general, we assume all triangulated categories are small; this is only required to ensure the existence of the virtual objects (see \cite{breuning2011determinant} and \cite{muro2010determinant}).
\begin{note}
  For any category $\mathcal{C}$, $\mathcal{C}(x;y)$ will stand as shorthand for $\Hom_{\mathcal{C}}(x;y)$. Similarly, if $\mathcal{C}$ is a multicategory, we use $\mathcal{C}(x_1,\dots,x_n;y)$ for the multi-morphisms. (Note that $0$-ary morphisms, i.e.\ morphisms with no inputs, are possible.)
\end{note}

We begin by clarifying the notion of multiexact functor for both Picard groupoids and triangulated categories.

\begin{defn}
  \label{multiexact_picard}
  For Picard groupoids $\mathcal{P}$, $\mathcal{Q}$ and $\mathcal{R}$, we will call the functor $F\colon \mathcal{P}\times\mathcal{Q}\to\mathcal{R}$ \emph{biexact} if it is monoidal in each variable and such that for all objects $a,b$ of $\mathcal{P}$ and $c,d$ of $\mathcal{Q}$ the diagram
  \begin{equation*}
    \begin{tikzcd}[sep=small]
      & F(a,c+d)+F(b,c+d) \arrow{rd} &\\
      && (F(a,c)+F(a,d))+(F(b,c)+F(b,d)) \arrow{dd}{\cong}[swap]{\commutassoc}\\
      F(a+b,c+d) \arrow{ruu} \arrow{rdd} &&\\
      && (F(a,c)+F(b,c))+(F(a,d)+F(b,d))\\
      & F(a+b,c)+F(a+b,d) \arrow{ru} &
    \end{tikzcd}
  \end{equation*}
  commutes. Similarly, for Picard groupoids $\mathcal{P}_1$, $\mathcal{P}_2$, $\dots$, $\mathcal{P}_n$ and $\mathcal{P}$, the functor $F\colon \mathcal{P}_1\times\dots\times\mathcal{P}_n\to\mathcal{P}$ is called \emph{multiexact} if it is a monoidal functor in each variable and satisfies the above diagram for any pair of variables.
\end{defn}

\begin{rmk}
  The coherence condition for the $\commutassoc$ guarantees that no further conditions are necessary in the definition of a multiexact functor with $n\geq 3$ variables in Definition \ref{multiexact_picard} above.
\end{rmk}

\begin{rmk}
  The above notion of multiexact functor can be rephrased in terms of \emph{cubical complexes,} see Remark \ref{remark_cubical_multiexactness} below.
\end{rmk}

The situation for triangulated categories is very similar and the definition of bifunctor is essentially standard (see, e.g.\ ref.\ \cite{KS2006}). The generalization to $n$ variables is also straightforward (\cite{schnurer2018sixop}). The definition is as follows:
\begin{defn}\label{def:multiexact_for_triangulated_categories}
  For triangulated categories $\mathcal{C}$, $\mathcal{D}$ and $\mathcal{T}$, we call the functor $F\colon \mathcal{C}\times\mathcal{D}\to\mathcal{T}$ \emph{biexact} if it is a triangulated functor in each variable. Specifically, if it preserves distinguished triangles in each variable, and preserves translation in each variable, i.e. $F(\Sigma c,d)\cong \Sigma F(c,d)$ and $F(c,\Sigma d)\cong \Sigma F(c,d)$, such that
  \begin{equation*}
    \begin{tikzcd}
      F(\Sigma c,\Sigma d) \arrow{d}[swap]{\cong} \arrow[r,"\cong"] \arrow[dr,phantom,"\scalebox{0.75}{-1}"] & \Sigma F(\Sigma c, d) \arrow[d,"\cong"]\\
      \Sigma F(c,\Sigma d) \arrow{r}[swap]{\cong} & \Sigma^2 F(c,d)
    \end{tikzcd}
  \end{equation*}
  is anti-commutative. Similarly, for triangulated categories $\mathcal{T}_1$, $\mathcal{T}_2$, $\dots$, $\mathcal{T}_n$ and $\mathcal{T}$, the functor $F\colon \mathcal{T}_1\times\dots\times\mathcal{T}_n\to\mathcal{T}$ is called \emph{multiexact} if it is a triangulated functor in each variable, and biexact for any pair of variables.
\end{defn}
  
As a consequence, a biexact functor $F\colon \mathcal{C}\times\mathcal{D}\to\mathcal{T}$ between triangulated categories along with two distinguished triangles $c'\to c\to c''\to \Sigma c'\in\mathcal{C}$ and $d'\to d\to d''\to\Sigma d'\in\mathcal{D}$ yields a commutative diagram
\begin{equation*}
  \begin{tikzcd}
    F(c',d')\arrow[r]\arrow[d] & F(c',d)\arrow[r]\arrow[d] & F(c',d'')\arrow[r]\arrow[d] & \Sigma F(c',d')\arrow[d]\\
    F(c,d')\arrow[r]\arrow[d] & F(c,d)\arrow[r]\arrow[d] & F(c,d'')\arrow[r]\arrow[d] & \Sigma F(c,d')\arrow[d]\\
    F(c'',d')\arrow[r]\arrow[d] & F(c'',d)\arrow[r]\arrow[d] & F(c'',d'')\arrow[r]\arrow[d]\arrow[dr,phantom,"\scalebox{0.75}{-1}"] & \Sigma F(c'',d')\arrow[d]\\
    \Sigma F(c',d')\arrow[r] & \Sigma F(c',d)\arrow[r] & \Sigma F(c',d'')\arrow[r] & \Sigma^2 F(c,d')
  \end{tikzcd}
\end{equation*}
of the type considered in ref.\ \cite[Proposition 1.1.11]{bbd1981pervers}. We shall return to these diagrams below in section \ref{sec:multi-determinant-funct-triangulated}. 

\begin{note}
  We will use $\TrCat$ to represent the category of small triangulated categories.  Similarly, we will use $\Pic$ for the category of Picard groupoids. 
\end{note}

It turns out that both these categories are in fact \emph{2-categories} in the obvious way. What is more, they can be promoted to $\Grpd$-enriched \emph{multicategories} by using multiexact functors and their natural isomorphisms as multi-morphisms. 

Recall that every monoidal category $(\mathcal{M},\otimes,I)$ gives rise to a multicategory, temporarily denoted $\mathcal{M}^\otimes$, with the same objects and multi-morphisms given by
\begin{equation*}
  \mathcal{M}^\otimes(x_1,\dots,x_n;y) = \mathcal{M}(x_1\otimes\dots \otimes x_n;y)\,.
\end{equation*}
(Some care will be needed in the right-hand side if the monoidal structure is non-strict.) It turns out this is the right-adjoint of an adjoint pair. The left adjoint assigns to every multicategory $\mathcal{M}$ a monoidal category, denoted $\mathcal{M}^\times$, whose objects are lists of objects from $\mathcal{M}$ with morphisms given by concatenation of those of $\mathcal{M}$
\begin{equation*}
  \mathcal{M}^\times ((x_1,\dots,x_n);(y_1,\dots,y_m)) = 
  \prod_{\phi\colon \mathbf{n}\to \mathbf{m}} 
  \mathcal{M}((x_i)_{i\in \phi^{-1}(j)};y_j)
\end{equation*}
where $\mathbf{n}={1,\dots,n}$ and $\phi\colon \mathbf{n}\to \mathbf{m}$ is an ``indexing" function  (see \cite{lurie2017higher,Leinster2004higher} and also \cite[Theorem 4.2]{elmendorf-mandell2009multi}). Below we will work with the monoidal (2-)categories $\TrCat^\times$ and $\Pic^\times$.
  
\begin{note}
  $\Pic$ will denote both the (enriched) category of Picard groupoids and the (enriched) multicategory of Picard groupoids. Therefore, according to this convention, $\TrCat(\mathcal{T}_1,\dots,\mathcal{T}_n;\mathcal{T})$ denotes the category of multiexact functors, and similarly for $\Pic$. In similar circumstances we will use the same name to represent both a category and the multicategory it induces. Note that $\Pic$ is closed.\footnote{It appears a complete, formal proof is not available in a single place. See \cite{elmendorf-mandell2009multi} and the more recent \cite{johnson2022homotopy,gurski2022symmetric}, which deal with the permutative case, from which one can deduce the statements for $\Pic$. The main point is the definition of the internal object and handling the inverses. Some of this can be found in \cite{aldrovandi2015biex}.}\footnote{Thanks to N.~Johnson, private communication.} (This is true for both meanings of $\Pic$.)
\end{note}

\section{Multi-determinant Functors for Triangulated Categories}
\label{sec:multi-determinant-funct-triangulated}

In this section we define multi-determinant functors. We provide two different, but equivalent, definitions: the first (in subsection~\ref{sec:definition_multideterminant}) is a direct generalization of Breuning's determinant functor; the second (in subsection~\ref{sec:cubical_complexes}) is based on categories of cubes in Picard groupoids.

\subsection{Definition of Multi-determinant Functor}
\label{sec:definition_multideterminant}

\begin{note}
When dealing with functions of multiple variables, for notational convenience, we will not show the variables that do not change. For example, the morphism
\begin{multline*}
    \det(a_1,\dots,a_{i-1},z_i,a_{i+1},\dots,a_n)
    \picadd
    \det(a_1,\dots,a_{i-1},x_i,a_{i+1},\dots,a_n) \to \\
    \to\det(a_1,\dots,a_{i-1},y_i,a_{i+1},\dots,a_n)  
\end{multline*}
will be written
  \begin{equation*}
    \det(z_i)+\det(x_i)\to\det(y_i)\,.
  \end{equation*}
\end{note}
Let $\mathcal{T}_1,\dots,\mathcal{T}_n$ be triangulated categories and $\mathcal{P}$ a Picard groupoid.
\begin{defn}\label{multi-det_functor}
  A \emph{multi-determinant functor}
  \begin{equation*}
    \GenericDet\colon \mathcal{T}_1\times\dots\times\mathcal{T}_n\to\mathcal{P}
  \end{equation*}
  is an $n$-functor
  \begin{equation*}
    \GenericDet\colon \isom(\mathcal{T}_1)\times\dots\times\isom(\mathcal{T}_n)\to\mathcal{P}
  \end{equation*}
  that is a determinant functor in each variable and satisfies the following compatibility conditions:
  \begin{enumerate}
  \item \emph{(Two Triangles Axiom)} For any distinct $i$ and $j$, and distinguished triangles
  \begin{align*}
    \Delta_i &: x_i\to y_i\to z_i\to \Sigma x_i\in \mathcal{T}_i\\
    \Delta_j &: x_j\to y_j\to z_j\to \Sigma x_j\in \mathcal{T}_j \,,
  \end{align*}
  the diagram below (pictured for $i<j$) commutes
  \begin{equation*}
    \begin{tikzcd}
&\GenericDetInput{y_i,z_j}+\GenericDetInput{y_i,x_j} \arrow{rdd}{\GenericDetInput{\Delta_j}} &\\
(\GenericDetInput{z_i,z_j}+\GenericDetInput{x_i,z_j})+(\GenericDetInput{z_i,x_j}+\GenericDetInput{x_i,x_j}) \arrow{dd}{\cong}[swap]{\commutassoc} \arrow{ru}{\GenericDetInput{\Delta_i}+\GenericDetInput{\Delta_i}} && \\
&&\GenericDetInput{y_i,y_j}\\
(\GenericDetInput{z_i,z_j}+\GenericDetInput{z_i,x_j})+(\GenericDetInput{x_i,z_j}+\GenericDetInput{x_i,x_j}) \arrow{rd}[swap]{\GenericDetInput{\Delta_j}+\GenericDetInput{\Delta_j}} &&\\
      &\GenericDetInput{z_i,y_j}+\GenericDetInput{x_i,y_j} \arrow{ruu}[swap]{\GenericDetInput{\Delta_i}}&
    \end{tikzcd}
  \end{equation*}

  \item \emph{(Triangle-Function Axiom)} For any distinct $i$ and $j$, morphism $f_i\colon a_i\to b_i$ in $\isom(\mathcal{T}_i)$ and distinguished triangle $\Delta_j \colon x_j\to y_j\to z_j\to \Sigma x_j$ in $\mathcal{T}_j$, the diagram below (pictured for $i<j$) commutes
  \begin{equation*}
    \begin{tikzcd}[sep=large]
      \GenericDetInput{a_i,z_j}+\GenericDetInput{a_i,x_j} \arrow[d, "\GenericDetInput{f_i}", swap] \arrow[r, "\GenericDetInput{\Delta_j}"] & \GenericDetInput{a_i,y_j} \arrow[d, "\GenericDetInput{f_i}"]\\
      \GenericDetInput{b_i,z_j}+\GenericDetInput{b_i,x_j} \arrow[r, "\GenericDetInput{\Delta_j}", swap] & \GenericDetInput{b_i,y_j}
    \end{tikzcd}
  \end{equation*}
  \end{enumerate}
\end{defn}

\begin{defn}\label{multi-det morph}
  Let $\mathcal{T}_i$ be triangulated categories and $\mathcal{P}$ be a Picard groupoid. For two multi-determinant functors $d_i\colon\mathcal{T}_1\times\dots\times\mathcal{T}_n\to\mathcal{P}$, a \emph{morphism between the multi-determinant functors} $\theta \colon d_1\Rightarrow d_2$ is a natural transformation that is compatible with the additivity data in each variable.
\end{defn}

In other words, $\theta$ is a natural transformation (on the isomorphism classes) in each variable such that for any distinguished triangle $\Delta_i\colon x_i\to y_i\to z_i\to\Sigma x_i$ in $\mathcal{T}_i$ the diagram below commutes.
\begin{equation*}
  \begin{tikzcd}
    d_1(z_i)+d_1(x_i)\arrow[r,"d_1(\Delta_i)"] \arrow{d}[swap]{\theta_{z_i}+\theta_{x_i}} & d_1(y_i) \arrow[d,"\theta_{y_i}"]\\
    d_2(z_i)+d_2(x_i) \arrow{r}[swap]{d_2(\Delta_i)} & d_2(y_i)
  \end{tikzcd}
\end{equation*}

For any triangulated categories $\mathcal{T}_1,\dots,\mathcal{T}_n$ and Picard groupoid $\mathcal{P}$, we will use $\Det(\mathcal{T}_1,\dots,\mathcal{T}_n;\mathcal{P})$ to denote the category whose objects are multi-determinant (or determinant) functors $\mathcal{T}_1\times\dots\times\mathcal{T}_n\to\mathcal{P}$ and whose morphisms are morphisms of multi-determinant functors (as defined above).  Note that $\Det(\mathcal{T}_1,\dots,\mathcal{T}_n;\mathcal{P})$ is a \emph{groupoid.}

\begin{note}[Virtual Objects] 
  For triangulated categories $\mathcal{T}_1,\dots,\mathcal{T}_n$, the morphisms $\det\colon\mathcal{T}_i\to V(\mathcal{T}_i)$ will denote universal determinant functors. Additionally, we will abuse notation and use $\det$ for the morphism comprised of their products and write $\det\colon\mathcal{T}_1\times\dots\times\mathcal{T}_n\to V(\mathcal{T}_1)\times\dots\times V(\mathcal{T}_n)$ instead of $\det\colon\isom(\mathcal{T}_1)\times\dots\times\isom(\mathcal{T}_n)\to V(\mathcal{T}_1)\times\dots\times V(\mathcal{T}_n)$.
\end{note}

\subsection{Cubes}\label{sec:cubes}

Let $I$ be the category $\{-1\to 0 \to 1\}$. In other words, $I$ is a category with three distinct objects: an initial object, a terminal object, and one more object.

We take $I^0$ to be a category with a single object and no nontrivial morphisms. For any integer $n\geq 1$, $I^n$ is a category with objects $(a_1,a_2,\dots,a_n)$, $a_i\in\{-1,0,1\}$, and non-identity morphisms based on $I$; to be more specific, for $i = 1,\dots,n$ there are non-identity morphisms
\begin{equation*}
  (a_1,\dots,a_{i-1},-1,a_{i+1},\dots,a_n)\to (a_1,\dots,a_{i-1},0,a_{i+1},\dots,a_n)
\end{equation*}
and 
\begin{equation*}
  (a_1,\dots,a_{i-1},0,a_{i+1},\dots,a_n)\to (a_1,\dots,a_{i-1},1,a_{i+1},\dots,a_n).
\end{equation*}
Lastly, we will use $\objects(I^n)$ to denote the category with the same objects as $I^n$ but no non-identity morphisms.

\subsubsection{Cubes in Picard groupoids}

We introduce, for any Picard groupoid $\mathcal{P}$, a cubical complex reminiscent of the Eilenberg-Mac~Lane's $Q$-complex for rings \cite{maclane1956ann}, and of subsequent adaptations to exact categories \cite{burgos2000arith,mccarthy1996chain}. Our complex is based on a system of Picard groupoids $C^n(\mathcal{P})$, for $n \geq 0$, of ``cubical shape,'' such that $C^0(\mathcal{P})=\mathcal{P}$, which we now define.

\begin{defn}
  \label{def_cube_Picard}
  For a Picard groupoid $\mathcal{P}$, an \emph{$n$-cube in $\mathcal{P}$} is a functor $S\colon \objects(I^n)\to \mathcal{P}$ along with isomorphisms 
  \begin{multline}
    \label{eq:cube_structure_morphism}
    f_i(a_0,\dots,a_{i-1},a_{i+1},\dots,a_n)\colon 
    S(a_0,\dots,a_{i-1},1,a_{i+1},\dots,a_n) + 
    S(a_0,\dots,a_{i-1},-1,a_{i+1},\dots,a_n) \\
    \xrightarrow{\cong}
    S(a_0,\dots,a_{i-1},0,a_{i+1},\dots,a_n)
  \end{multline}
  for $i=1,\dots,n$, such that for each pair $(i,j)$, $0\leq i < j \leq n$, the following diagram commutes
  \begin{equation}
    \label{eq:cube_pentagon}
    \begin{tikzcd}[column sep=6em,row sep=large]
      (S(1,1)\picadd S(-1,1)) \picadd (S(1,-1)\picadd S(-1,-1))
      \arrow[dd,"\cong","\commutassoc"'] 
      \ar[r,"\scriptstyle{ f_i(1) \picadd f_i(-1)}"] &
      S(0,1) \picadd S(0,-1) \ar[d,"f_j(0)"] \\ 
      & S(0,0) \\
      (S(1,1)\picadd S(1,-1)) \picadd (S(-1,1)\picadd S(-1,-1)) 
      \ar[r,"\scriptstyle{f_j(1) \picadd f_j(-1)}"'] &
      S(1,0) \picadd S(-1,0) \ar[u,"f_i(0)"']
    \end{tikzcd}
  \end{equation}
  (To ease the notation, we did not write the $a_k$'s that stay constant.) It is evident that the $n$-cubes form a category, in fact a groupoid, denoted $C^n(\mathcal{P})$, where morphisms are natural transformations of functors.
\end{defn}
\begin{rmk}
  The notion of cube in Definition \ref{def_cube_Picard} is valid for any symmetric monoidal category, not just Picard groupoid. However, we need only consider the latter in this paper.
\end{rmk}
\begin{ex}
  \label{example_1-cube}
  A $1$-cube in a Picard groupoid $\mathcal{P}$ is a choice of two elements $x$ and $z$, a choice for their sum called $y$, and a choice of isomorphism $x+z\to y$. This is analogous to the description of categories with sums by Segal (cf.\ \cite{segal1974catcoho}).
\end{ex}
% Moreover, 
\begin{note}
  It is convenient to draw cubes by emphasizing the shape arising from the combinatorics of the categories $I^n$. For example, a $1$-cube in $\mathcal{P}$ can be visualized as
  \begin{equation*}
    \left\lbrace
    \begin{tikzcd}[sep=small]
      -1 & 0 & 1
    \end{tikzcd}
    \right\rbrace 
    \longmapsto
    \begin{tikzcd}
      x \arrow[r, dash] & y \arrow[r, dash] & z
    \end{tikzcd},
  \end{equation*}
  whereas for $n=2$ the $2$-cube $S\colon I^2 \to \mathcal{P}$ is drawn as
  \begin{equation*}
    \left\lbrace
    \begin{tikzcd}[sep=small]
      (-1,1) & (0,1) & (1,1) \\
      (-1,0) & (0,0) & (1,0) \\
      (-1,-1) & (0,-1) & (1,-1)
    \end{tikzcd}\right\rbrace 
    \longmapsto 
    \begin{tikzcd}[sep=large]
      b \arrow[r, dash] \arrow[d, dash] & p \arrow[r, dash] \arrow[d, dash] & a \arrow[d, dash]\\
      s \arrow[r, dash] \arrow[d, dash] & z \arrow[r, dash] \arrow[d, dash] & r \arrow[d, dash]\\
      d \arrow[r, dash] & q \arrow[r, dash] & c
    \end{tikzcd},
  \end{equation*}        
  where the lines in the image showcase what elements are ``related'' via the isomorphisms, which form the commutative diagram
  \begin{equation*}
    \begin{tikzcd}
      (a\picadd b) \picadd (c\picadd d) \arrow{rr}{\cong}[swap]{\commutassoc} \arrow{d} && (a\picadd c)\picadd (b\picadd d) \arrow{d}\\
      p\picadd q \arrow{r} & z & r\picadd s \arrow{l}
    \end{tikzcd}
  \end{equation*}
  as per diagram \eqref{eq:cube_pentagon} in Definition \ref{def_cube_Picard} above.
\end{note}

One may wonder whether additional compatibility conditions, beyond that in Definition \ref{def_cube_Picard}, should be considered. Indeed, we note that there are additional commutative diagrams, i.e. compatibility conditions, for $n$-cubes when $n\geq 3$. For example, if $S\colon I^3\to \mathcal{P}$ is a $3$-cube, we may draw a (large) diagram resulting from the decomposition of the center vertex $S(0,0,0)$ of the cube using the three possible instances of \eqref{eq:cube_pentagon}.  However, the higher compatibility conditions are automatically satisfied for Picard groupoids, and in fact for symmetric monoidal categories—this is a result of the first section of \cite{iteratedmonoidal}. More specifically, we have 
\begin{prop}
  Let $\mathcal{C}$ be a symmetric monoidal category. For any cube $S\colon I^n\to \mathcal{C}$ and any positive integer $k$ such that $3\leq k\leq n$, all diagrams constructed using \eqref{eq:cube_structure_morphism}, starting from the central vertex $S(0,\dots,0)$ in a manner analogous to \eqref{eq:cube_pentagon}, commute.
\end{prop}
\begin{proof}
  Following the ideas of \cite{iteratedmonoidal}, if we start from the case of a $3$-cube $S$ and proceed in the manner just alluded from the center vertex $S(0,0,0)$, the resulting diagram is the same as the ``big'' one in \cite[p.\ 284]{iteratedmonoidal}, when all monoidal structures are the same, or as the one appearing in \cite[App.\ D]{aldrovandi2015biex}. We write it—schematically—as:
  \begin{equation*}
    \begin{tikzcd}[row sep=large, font=\small]
      ((a_1 \picadd a_2) \picadd (b_1 \picadd b_2)) \picadd
      ((c_1 \picadd c_2) \picadd (d_1 \picadd d_2))
      \arrow[d,"c\picadd c"'] 
      \arrow[r,"c"] &
      ((a_1 \picadd a_2) \picadd (c_1 \picadd c_2)) \picadd
      ((b_1 \picadd b_2) \picadd (d_1 \picadd d_2))
      \arrow[d,"c\picadd c"] \\
      ((a_1 \picadd b_1) \picadd (a_2 \picadd b_2)) \picadd
      ((c_1 \picadd d_1) \picadd (c_2 \picadd d_2))
      \arrow[d,"c"'] &
      ((a_1 \picadd c_1) \picadd (a_2 \picadd c_2)) \picadd
      ((b_1 \picadd d_1) \picadd (b_2 \picadd d_2))
      \arrow[d,"c"] \\
      ((a_1 \picadd b_1) \picadd (c_1 \picadd d_1)) \picadd
      ((a_2 \picadd b_2) \picadd (c_2 \picadd d_2))
      \arrow[r,"c\picadd c"'] &
      ((a_1 \picadd c_1) \picadd (b_1 \picadd d_1)) \picadd
      ((a_2 \picadd c_2) \picadd (b_2 \picadd d_2))
    \end{tikzcd}
  \end{equation*}
  This diagram is the coherence for the morphism we call $\commutassoc$, subsuming Mac~Lane's pentagon and hexagon diagrams, and it commutes if and only if the monoidal structure on $\mathcal{C}$ is symmetric (see \cite{aldrovandi2015biex} for an explicit proof).  All the other cells in the total diagram we get starting from $S(0,0,0)$ are either squares, which commute by functoriality of $\commutassoc$, or instances of \eqref{eq:cube_pentagon} and sums thereof, which also commute. This takes care of the case $n=k=3$, and all other cases follow from this and the results in the first section of \cite{iteratedmonoidal}.
\end{proof}

\begin{rmk}
  The first section of \cite{iteratedmonoidal} additionally implies that an $n$-cube in Picard groupoid $\mathcal{P}$ can be recursively defined: The set of all $1$-cubes in $\mathcal{P}$ is itself a Picard groupoid. Thus, we can take a $1$-cube in the set of all $1$-cubes in $\mathcal{P}$---this yields a $2$-cube in $\mathcal{P}$.  Continuing recursively, we can define an $n$-cube in $\mathcal{P}$ to be a $1$-cube in the set of all $(n-1)$-cubes in $\mathcal{P}$.
\end{rmk}

\begin{rmk}
  In light of the previous remark, and again by \cite{iteratedmonoidal}, the existence of the tower
  \begin{equation*}
    C^0(\mathcal{P}) = \mathcal{P}\,,\quad
    C^1(\mathcal{P}) \,, \quad
    C^2(\mathcal{P}) \,, \quad \dotsm
  \end{equation*}
  is \emph{equivalent} to $\mathcal{P}$ carrying a symmetric monoidal structure. This is valid for symmetric monoidal categories in general. 
\end{rmk}

\subsubsection{Cubes in triangulated categories}
Next, we move to defining cubes in triangulated categories. However, we will focus on lower dimensional cubes because triangulated categories do not have enough guaranteed structure for higher dimensional cubes.

Let us consider the $3\times 3$\footnote{Of course, the diagram should more properly regarded as being $4\times 4$, but the rightmost column and the bottom row are determined by the rest, so it is easier to consider it as a $3\times 3$ one.} commutative diagram \cite{bbd1981pervers,may2001ttt}
\begin{equation}
  \label{eq:9x9}
  \begin{tikzcd}
    x'\arrow[r,"f'"]\arrow[d,"p'"] & y' \arrow[r,"g'"]\arrow[d,"q'"] & 
    z' \arrow[r,"h'"]\arrow[d,"r'"] & \Sigma x'\arrow[d,"\Sigma p'"]\\
    x\arrow[r,"f"]\arrow[d,"p"] & y\arrow[r,"g"]\arrow[d,"q"] & 
    z\arrow[r,"h"]\arrow[d,"r"] & \Sigma x\arrow[d,"\Sigma p"]\\
    x''\arrow[r,"f''"]\arrow[d,"p''"] & y''\arrow[r,"g''"]\arrow[d,"q''"] & 
    z''\arrow[r,"h''"]\arrow[d,"r''"]\arrow[dr,phantom,"\scalebox{0.75}{-1}"] & \Sigma x''\arrow[d,"\Sigma p''"]\\
    \Sigma x'\arrow[r,"\Sigma f'"] & \Sigma y'\arrow[r,"\Sigma g'"] & 
    \Sigma z'\arrow[r,"\Sigma h'"] & \Sigma^2 x'
  \end{tikzcd}
\end{equation}
(also called a 9 term diagram) whose rows and columns are distinguished triangles.

\begin{defn}[\cite{bbd1981pervers,may2001ttt}]
  \label{Verdier_Structure}
  In a triangulated category $\mathcal{T}$, the diagram \eqref{eq:9x9} admits a \emph{Verdier structure} when there exists an object $A$ of $\mathcal{T}$ and octahedrons
  \begin{gather*}
    \begin{tikzcd}[ampersand replacement=\&]
      x' \arrow[r,equal] \arrow[d,"p'"]\& 
      x' \arrow[d,"f\circ p'"] \\
      x \arrow[r,"f"] \arrow[d,"p"] \& 
      y \arrow[r,"g"] \arrow[d,"\alpha"] \&
      z \arrow[r,"h"] \arrow[d, equal] \&
      \Sigma x \arrow[d,"\Sigma p"] \\
      x'' \arrow[r,"\epsilon"] \arrow[d,"p''"] \&
      A \arrow[r,"\eta"] \arrow[d,"v"] \&
      z \arrow[r,"\zeta"] \& \Sigma x'' \\
      \Sigma x' \arrow[r, equal] \& \Sigma x'
    \end{tikzcd}
    \qquad
    \begin{tikzcd}[ampersand replacement=\&]
      x' \arrow[r,"f'"] \arrow[d, equal] \& 
      y' \arrow[r,"g'"] \arrow[d,"q'"] \& 
      z' \arrow[r,"h'"] \arrow[d,"\beta"] \& 
      \Sigma x' \arrow[d, equal] \\
      x' \arrow[r,"q'\circ f'"] \& 
      y \arrow[r,"\alpha"] \arrow[d,"q"] \& 
      A \arrow[r,"u"] \arrow[d,"\gamma"] \& \Sigma x' \\
      \& y'' \arrow[r, equal] \arrow[d,"q''"] \& y'' \arrow[d,"\delta"]\\
      \& \Sigma y' \arrow[r,"\Sigma g'"] \& \Sigma z'
    \end{tikzcd}
    \\
    \begin{tikzcd}[ampersand replacement=\&]
        x'' \arrow[r,"\epsilon"] \arrow[d, equal] \& 
        A \arrow[r,"\eta"] \arrow[d,"\gamma"] \& 
        z \arrow[r,"\zeta"] \arrow[d,"r"] \& \Sigma x'' \arrow[d, equal] \\
        x'' \arrow[r,"f''"] \& y'' \arrow[r,"g''"] \arrow[d,"\delta"] \& 
        z'' \arrow[r,"h''"] \arrow[d,"r''"] \& \Sigma '' \\
        \& \Sigma z' \arrow[r, equal] \arrow[d,"\Sigma \beta"] \& 
        \Sigma z' \arrow[d,"\Sigma r'"]\\
        \& \Sigma A \arrow[r,"\Sigma\eta"] \& \Sigma z
    \end{tikzcd}
  \end{gather*}
\end{defn}

\begin{rmk}
  From ref.\ \cite[Prop. 1.1.11]{bbd1981pervers} (see also \cite{may2001ttt}) we have that every diagram of the form
  \begin{equation*}
    \begin{tikzcd}
      x'\arrow[r,"f'"]\arrow[d,"p'"] & y' \arrow[r,"g'"]\arrow[d,"q'"] & z' \\
      x\arrow[r,"f"]\arrow[d,"p"] & y\arrow[r,"g"]\arrow[d,"q"] & z\\
      x'' & y'' &
    \end{tikzcd}
  \end{equation*}
  can be completed to one of the form \eqref{eq:9x9} by combining three octahedra as in Definition \ref{Verdier_Structure}. The resulting 9 term diagram will thus have a Verdier structure.
  Given an arbitrary $3 \times 3$ diagram, we can forget the bottom and rightmost maps to get a diagram of the above form and \emph{then} apply the construction of \cite{bbd1981pervers,may2001ttt} to reconstruct a new $3\times 3$ diagram with a Verdier structure.
  Note that this newly constructed $3\times 3$ diagram might differ from the original by having an isomorphic object at the location $z''$. It's also clear these diagrams are non-unique.
\end{rmk}
\begin{defn}
  For $n = 0, 1, 2$ and a triangulated category $\mathcal{T}$, an \emph{$n$-cube in $\mathcal{T}$} is: 
  \begin{enumerate}
    \item an object $x$ of $\mathcal{T}$ if $n=0$;
    \item a distinguished triangle $\Delta$ if $n=1$; and
    \item a diagram \eqref{eq:9x9} \emph{equipped with a Verdier structure} if $n=2$.
  \end{enumerate}
  The $n$-cubes in $\mathcal{T}$ form a groupoid by taking object-wise isomorphisms. We denote it by $C^n(\isom \mathcal{T})$, and note that $C^0(\mathcal{T})=\isom (\mathcal{T})$.
\end{defn}

\begin{rmk}
  The octahedron
  \begin{equation*}
    \begin{tikzcd}
      x \arrow[r] \arrow[d, equal] & y \arrow[r] \arrow[d] & 
      z' \arrow[r] \arrow[d] & \Sigma x \arrow[d, equal] \\
      x \arrow[r] & z \arrow[r] \arrow[d] & 
      y' \arrow[r] \arrow[d] & \Sigma x \\
      & x' \arrow[r, equal] \arrow[d] & x' \arrow[d] \\
      & \Sigma y \arrow[r] & \Sigma z'
    \end{tikzcd}
  \end{equation*}
  can be written as a $2$-cube
  \begin{equation*}
    \begin{tikzcd}
      x \ar[r] \ar[d,"\id"] & y \ar[r] \ar[d] &
      z' \ar[r] \ar[d] & \Sigma x \ar[d,"\id"] \\
      x \ar[r] \ar[d] & z \ar[r] \ar[d] &
      y' \ar[r] \ar[d] & \Sigma x \ar[d] \\
      0 \ar[r] \ar[d] & x' \ar[r,"\id"] \ar[d] &
      x' \ar[r] \ar[d] & 0 \ar[d] \\
      \Sigma \ar[r] x & \Sigma y \ar[r] & 
      \Sigma z' \ar[r] & \Sigma^2 x 
    \end{tikzcd}
  \end{equation*}
  with an evident Verdier structure (relative to Definition \ref{Verdier_Structure}) with $A=y'$ and in which one of the octahedra is just the one we start from, and the other (trivial) two are:
  \begin{equation*}
    \begin{tikzcd}
      x \ar[r,equal] \ar[d,"\id"] & x \ar[d] \\
      x \ar[r] \ar[d] & z \ar[r] \ar[d] &
      y' \ar[r] \ar[d,equal] & \Sigma x \ar[d] \\
      0 \ar[r] \ar[d] & y' \ar[r,"\id"] \ar[d] &
      y' \ar[r] & 0 \\
      \Sigma x \ar[r] & \Sigma x
    \end{tikzcd}
    \qquad
    \begin{tikzcd}
      0 \ar[r] \ar[d,equal] & y' \ar[r,"\id"] \ar[d] &
      y' \ar[r] \ar[d] & 0 \ar[d,equal] \\
      0 \ar[r] & x' \ar[r,"\id"] \ar[d] &
      x' \ar[r] \ar[d] & 0 \\
      & \Sigma z' \ar[r,equal] \ar[d] & \Sigma z' \ar[d] \\
      & \Sigma y' \ar[r,"\id"] & \Sigma y' \\
    \end{tikzcd}
  \end{equation*}
\end{rmk}

\begin{rmk}
  In triangulated categories, $n$-cubes ($n \leq 2$) can also be defined as functors from $I^n$ to $\mathcal{T}$ satisfying obvious additional conditions.
\end{rmk}

\subsection{Cubical Complexes}
\label{sec:cubical_complexes}

We use the term ``complex'' in a loose way, to denote diagrams comprised of the categories of cubical shapes introduced in the previous sections.

Recall from \cite{loday2982finitely,burgos2000arith} that the categories $I^n$ come with functors $\delta_\alpha^j \colon I^{n-1} \to I^n$, for $\alpha=-1,0,1$ and $j=1,\dots,n$, corresponding to the inclusion in the $j^\text{th}$-direction in one of the three possible positions $-1,0,1$. Let $\mathcal{C}$ be a symmetric monoidal category. Recall the groupoids $C^n(\mathcal{C})$ of $n$-cubes $S\colon I^n\to \mathcal{C}$ (a shorthand for $S\colon \objects(I^n) \to\mathcal{C}$). 

Composing with the $\delta_\alpha^j$ we get ``face maps''
\begin{equation}
  \label{eq:cubical_face}
  \partial^\alpha_j \colon C^n(\mathcal{C}) \to C^{n-1}(\mathcal{C})\,;
  \qquad
  j=1,\dots,n\,,\; \alpha=-1,0,1\,.
\end{equation}
Thus, if $S$ is an $n$-cube, we have
\begin{equation*}
  \partial^\alpha_{j}(S)(a_1,\dots,a_{n-1}) = S(a_1,\dots,a_{j-1},\alpha,a_j,\dots,a_{n-1})
\end{equation*}
There are also ``degeneracy maps'' going in the opposite direction:
\begin{equation}
  \label{eq:cubical_degeneracy}
  \begin{gathered}
    s_\alpha^j \colon C^n(\mathcal{C}) \to C^{n+1}(\mathcal{C})\,;
    \qquad 
    j=1,\dots,n\,,\; \alpha = -1,1\,,\\
    s_\alpha^j (S) (a_1,\dots,a_{n+1}) = 
    \begin{cases}
      S(a_1,\dots,a_{j-1},a_{j+1}\dots,a_{n+1}) & a_j \neq \alpha \\
      0 & a_j = \alpha \,.
    \end{cases}
  \end{gathered}
\end{equation}
(Note that these ``degeneracy maps'' are not induced by corresponding functors $I^n\to I^{n-1}$.)  Both face and degeneracy ``maps'' are actually functors, since they evidently are compatible with morphisms of cubes (which are defined pointwise).  We also have the relations:
\begin{equation}
  \label{eq:cubical_relations}
  \begin{aligned}
    \partial^\alpha_i\partial^\beta_j &= \partial^\beta_{j-1}\partial^\alpha_i 
    \quad
    i < j\,;\; \alpha,\beta = -1,0,1 \\
    s^j_\beta s^i_\alpha &= s^i_\alpha s^{j-1}_\beta 
    \quad
    i < j\,;\; \alpha,\beta = -1,1 \\
    \partial^\alpha_i s^j_\beta &= 
    \begin{cases}
      s^{j-1}_\beta \partial^\alpha_i & i < j\,;\;\alpha = -1,0,1\,, \beta = \pm 1 \\
      \id & i = j\,;\; \alpha = -1,0,1\,, \beta = \pm 1\,, \alpha \neq \beta \\
      0 & i = j\,;\; \alpha = \beta = \pm 1 \\
      s^j_\beta \partial^\alpha_{i-1} & i > j\,;\;\alpha = -1,0,1\,, \beta = \pm 1 
    \end{cases}
  \end{aligned}
\end{equation}
The maps $\partial^\alpha_i$ and $s^i_\beta$ and the relations they satisfy still make sense for the objects $C^k(\mathcal{T})$, where $\mathcal{T}$ is a triangulated category, introduced above, with the limitation that $k\leq 2$.
\begin{defn}
  \label{cubical_complex}
  Let $\mathcal{C}$ be a category for which a definition of $n$-cubes for all $n$ (or for $n=0,\dots,k$) makes sense. The \emph{cubical complex on $\mathcal{C}$}, denoted $C^\ast(\mathcal{C})$ (or $C^{\ast\leq k}(\mathcal{C})$, respectively), is the diagram of groupoids:
  \begin{equation*}
    \begin{tikzcd}
      \mathcal{C} = C^0(\mathcal{C}) 
      \arrow[r,shift left=2ex]
      \arrow[r,shift left=1ex] &
      C^1(\mathcal{C}) 
      \arrow[r,shift left=4ex]
      \arrow[r,shift left=2ex,phantom,"\dotsi"{font=\tiny}]
      \arrow[r,shift left=1ex]
      \arrow[l] \arrow[l,shift left=1ex] \arrow[l,shift left=2ex]  &
      C^2(\mathcal{C}) 
      \arrow[l]
      \arrow[l,shift left=2ex,phantom,"\dotsi"{font=\tiny,swap}]
      \arrow[l,shift left=4ex]
      &
      \dotsi
    \end{tikzcd}
  \end{equation*}
\end{defn}
\begin{rmk}
  While we will only consider the cases $\mathcal{C}=\mathcal{P}$, a Picard groupoid, or $\mathcal{C}=\mathcal{T}$, a triangulated category, Definition \ref{cubical_complex} can be applied more generally. For example, $\mathcal{C}$ can be a category with sums and a zero object, or, more in keeping with the original idea, an exact category or even a Waldhausen one.
\end{rmk}
\begin{rmk}
  Despite the name, the diagram in Definition \ref{cubical_complex} is a diagram of groupoids, as opposed to a complex in the strict sense of the word. We can obtain from it an actual complex, which turns out to be the analog of Eilenberg-Mac~Lane's $Q$-construction, as follows (see \cite{burgos2000arith,mccarthy1996chain}). Let $F$ be a functor from groupoids to an abelian category $\mathcal{A}$. Define $Q'_n(\mathcal{C},F) = F(C^n(\mathcal{C}))$, and $\partial_n \colon Q'_n(\mathcal{C},F) \to Q'_{n-1}(\mathcal{C},F)$ by
  \begin{equation*}
    \partial_n = \sum_{i=1}^n \sum_{\alpha=-1}^1 (-1)^{i+\alpha+1} F(\partial^\alpha_i)\,.
  \end{equation*}
  It is easily verified that $\partial^2 = 0$, so that $Q'_\bullet(\mathcal{C},F)$ is a complex in $\mathcal{A}$. Furthermore, $\partial$ preserves the subobject generated by the images of $F(s^i_\alpha)$, so that we can define $Q_\bullet(\mathcal{C},F)$ as the quotient of $Q'_\bullet(\mathcal{C},F)$ by the degenerate subcomplex. 

  With $\mathcal{A}=\mathbf{Ab}$, the category of abelian groups, and $F(\mathcal{C}) = \ZZ[\objects(\mathcal{C})]$, this is literally the Eilenberg-Mac~Lane's $Q$-complex if $\mathcal{C}$ is an abelian group, and a direct generalization of it if $\mathcal{C}$ is a Picard groupoid \cite{mccarthy1996chain}. Analogously to the classical case, these complexes behave very well relative to multiexact maps $\mathcal{C}_{1}\times \dots \times \mathcal{C}_{n}\to \mathcal{C}'$, such as the multiplication map $m\colon \mathcal{C}\times \mathcal{C}\to \mathcal{C}$ and its higher associates when $\mathcal{C}$ is a categorical ring (see below, sect.~\ref{sec:det_ttt}). Such complexes will be relevant in a sequel to this paper, when we study the Mac~Lane's cohomology of the determinant functor $\mathcal{C}=V(\mathcal{T})$ of a tensor triangulated category $\mathcal{T}$ (again, see below). They will not be used in what follows.

\end{rmk}

\subsection{Alternative Definition of Multi-determinant Functor}
\label{sec:alternative_definition}

We introduce a secondary definition of multi-determinant functor, and prove in Theorem \ref{defs_of_det_are_equivalent} that our two definitions of multi-determinant functor are equivalent.

\begin{defn}\label{cubical_det}
  For a triangulated category $\mathcal{T}$ and Picard groupoid $\mathcal{P}$, a \emph{cubical determinant functor} is a morphism of diagrams
  \begin{equation*}
    D\colon C^{\leq 2}(\isom\mathcal{T})\to C^{\leq 2}(\mathcal{P}) \,.
  \end{equation*}
\end{defn}
\begin{rmk}
  \label{remark_symmetric_monoidal}
  Let $\mathcal{C}$ and $\mathcal{D}$ be two symmetric monoidal categories. To say that a functor $F \colon \mathcal{C} \to \mathcal{D}$ is symmetric monoidal is precisely the same as saying that it extends to a morphism of diagrams
  \begin{equation*}
    F \colon C^\ast(\mathcal{C}) \to C^\ast (\mathcal{D})\,,
  \end{equation*}
  where, in effect, we need only consider $\ast \leq 2$. Thus, all conditions characterizing $F$ as a symmetric monoidal functor are built in the notion of morphism of cubes.
\end{rmk}
The above definition can be generalized to several variables, beginning with the cube complex itself.
\begin{defn}
  For categories $\mathcal{C}_1,\dots,\mathcal{C}_m$, the \emph{cubical complex} $C^\ast(\mathcal{C}_1,\dots,\mathcal{C}_m)$ is the diagram whose $n$-th level is
  \begin{equation*}
    C^n(\mathcal{C}_1,\dots,\mathcal{C}_m) \coloneqq 
    \coprod_{\underset{n_i\geq 0}{\Sigma n_i = n}} 
    C^{n_1}(\mathcal{C}_1)\times\dots\times C^{n_m}(\mathcal{C}_m)  
  \end{equation*}
  with face and degeneracy maps induced from $C^\ast(\mathcal{C}_i)$, as defined above Definition \ref{cubical_complex}.  Similarly, $C^{\ast\leq k}(\mathcal{C}_1,\dots,\mathcal{C}_m)$ denotes the analogous diagram truncated below level $k$.
\end{defn}
\begin{rmk}
  \label{remark_cubical_multiexactness}
  Generalizing Remark \ref{remark_symmetric_monoidal}, let $\mathcal{C}_1,\dots,\mathcal{C}_n, \mathcal{D}$ be symmetric monoidal categories. An $n$-functor
  \begin{math}
    F \colon \mathcal{C}_1 \times \dots \times \mathcal{C}_n \to \mathcal{D}
  \end{math}
  is multiexact precisely when it extends to a morphism of diagrams
  \begin{equation*}
    F \colon C^\ast (\mathcal{C}_1 \times \dots \times \mathcal{C}_n)
    \to C^\ast(\mathcal{D})\,.
  \end{equation*}
  In other words, cubical diagrams and their morphisms encode multi-exactness.
\end{rmk}
\begin{defn}
  \label{cubical_multidet}
  For triangulated categories $\mathcal{T}_1,\dots,\mathcal{T}_m$ and Picard groupoid $\mathcal{P}$, a \emph{cubical multi-determinant functor} is a functor
  \begin{equation*}
    D\colon C^{\leq 2}(\isom(\mathcal{T}_1),\dots, \isom(\mathcal{T}_m))\to C^{\leq 2}(\mathcal{P})  
  \end{equation*}
  in the sense that it maps $n$-cubes to $n$-cubes and respects the face and degeneracy maps.
\end{defn}

\begin{lemma}
  \label{det_with_zero_as_an_input}
  For triangulated categories $\mathcal{T}_1,\dots,\mathcal{T}_m$ and Picard groupoid $\mathcal{P}$, if 
  \begin{equation*}
    D\colon C^{\leq 2}(\isom(\mathcal{T}_1),\dots, \isom(\mathcal{T}_m))\to C^{\leq 2}(\mathcal{P})
  \end{equation*} 
  is a cubical multi-determinant, then for $0$ in $\mathcal{T}_i$ and $a_k$ in $\mathcal{T}_k$, $k\in\{1,\dots,m\}-\{i\}$, we have
  \begin{equation*}
    D(a_1,\dots, a_{i-1},0,a_{i+1},\dots,a_m)\cong 0.
  \end{equation*}
\end{lemma}
\begin{proof}
  Since $0\to 0\to 0\to \Sigma 0$ is a distinguished triangle in $\mathcal{T}_i$, then
  \begin{multline*}
    (a_1,\dots, a_{i-1},0,a_{i+1},\dots,a_m)\to 
    (a_1,\dots, a_{i-1},0,a_{i+1},\dots,a_m)\to 
    (a_1,\dots, a_{i-1},0,a_{i+1},\dots,a_m) \\ \to
    \Sigma (a_1,\dots, a_{i-1},0,a_{i+1},\dots,a_m)
  \end{multline*}
  is a $1$-cube in $C^1(\isom(\mathcal{T}_1)\times\dots\times \isom(\mathcal{T}_m))$. Therefore, its image is a $1$-cube in $\mathcal{P}$, i.e.
  \begin{multline*}
    D(a_1,\dots, a_{i-1},0,a_{i+1},\dots,a_m) \picadd 
    D(a_1,\dots, a_{i-1},0,a_{i+1},\dots,a_m)\xrightarrow{\cong} \\
    \xrightarrow{\cong} D(a_1,\dots, a_{i-1},0,a_{i+1},\dots,a_m).
  \end{multline*}
  But all objects are invertible in $\mathcal{P}$, which gives the desired result.
\end{proof}

\begin{thm}
  \label{defs_of_det_are_equivalent}
  The two definitions of multi-determinant functors, from Definitions \ref{multi-det_functor} and \ref{cubical_multidet}, are equivalent.
\end{thm}
\begin{proof}
  Let $\mathcal{T}_1,\dots,\mathcal{T}_m$ be triangulated categories and $\mathcal{P}$ be a Picard groupoid.

  First, suppose $d\colon \mathcal{T}_1\times\dots\times\mathcal{T}_m\to \mathcal{P}$ is a multi-determinant functor in the sense of definition \ref{multi-det_functor}.
  Then $d$ is a morphism $C^0(\isom(\mathcal{T}_1),\dots, \isom(\mathcal{T}_m))\to C^0(\mathcal{P})$. We claim this morphism of level $0$ induces morphisms of the higher levels which are compatible with the face and degeneracy maps.

  Indeed, for level $1$, a $1$-cube in $C^1(\isom(\mathcal{T}_1),\dots, \isom(\mathcal{T}_m))$ is a distinguished triangle $x_i\to y_i\to z_i\to \Sigma x$ in $\mathcal{T}_i$ along with objects $a_j$ in $\mathcal{T}_j$ for $j = 1,\dots,i-1,i+1,\dots m$.  Since $d$ is a multi-determinant functor, it is a determinant functor in variable $i$. Therefore, suppressing the variables $a_j$ for notational convenience, we have 
  \begin{equation*}
    d(z_i)+d(x_i)\xrightarrow{\cong} d(y_i)\,,
  \end{equation*}
  that is
  \begin{equation*}
    d\bigl(
      \begin{tikzcd}[cramped,sep=small]
        x_i \ar[r,dash] & y_i \ar[r,dash] & z_i
      \end{tikzcd}
    \bigr) = 
    \begin{tikzcd}[cramped,sep=small]
      d(x_i) \ar[r,dash] & d(y_i) \ar[r,dash] & d(z_i)
    \end{tikzcd}
  \end{equation*}
  is a $1$-cube in $\mathcal{P}$. Furthermore, it is clear the face and degeneracy maps between levels 0 and 1 are compatible because things of the form 
  \begin{math}
    \begin{tikzcd}[cramped,sep=small]
      x \ar[r,equal] & x \ar[r] & 0 \ar[r] & \Sigma x
    \end{tikzcd}
  \end{math}
  and 
  \begin{math}
    \begin{tikzcd}[cramped,sep=small]
      0 \ar[r] & x \ar[r,equal] & x \ar[r] & 0 
    \end{tikzcd}
  \end{math}
  are always distinguished triangles.

  For level $2$, there are two types of $2$-cubes in $C^2(\isom(\mathcal{T}_1),\dots, \isom(\mathcal{T}_m))$: the first type is a distinguished triangle $x_i\to y_i\to z_i \to \Sigma x_i$ in $\mathcal{T}_i$ and a distinguished triangle $x_j\to y_j\to z_j\to \Sigma x_j$ in $\mathcal{T}_j$ for $i<j$ along with objects $a_k$ in $\mathcal{T}_k$ for $k \in \{1,\dots,m\}-\{i,j\}$;
  the second type is a $2$-cube in $C^2(\mathcal{T}_i)$ along with objects $a_k$ in $\mathcal{T}_k$ for $k \in \{1,\dots,m\}-\{i\}$.  Suppressing the $a_k$ for notational convenience, the image of the first type of $2$-cube under $d$ is
  \begin{equation*}
    \begin{tikzcd}
      d(z_i,x_j) \arrow[r, dash] \arrow[d, dash] & d(z_i,y_j) \arrow[r, dash] \arrow[d, dash] & d(z_i,z_j) \arrow[d, dash]\\
      d(y_i,x_j) \arrow[r, dash] \arrow[d, dash] & d(y_i,y_j) \arrow[r, dash] \arrow[d, dash] & d(y_i,z_j) \arrow[d, dash]\\
      d(x_i,x_j) \arrow[r, dash] & d(x_i,y_j) \arrow[r, dash] & d(x_i,z_j)
    \end{tikzcd}.
  \end{equation*}
  Notice this image is $2$-cube in $\mathcal{P}$ due to the two triangles axiom in the multi-determinant functor definition. As for the second type of $2$-cube, its image under $d$ is guaranteed to be a $2$-cube in $\mathcal{P}$ by \cite[Lemma 3.9]{breuning2011determinant}.

  Second, suppose $D\colon C^{\ast\leq 2}(\isom(\mathcal{T}_1),\dots, \isom(\mathcal{T}_m))\to C^{\ast\leq 2}(\mathcal{P})$ is a cubical multi-determinant functor in the sense of definition \ref{cubical_multidet}.  Then the level zero portion of $D$, which we will call $D^0$, defines a functor $\isom(\mathcal{T}_1)\times\dots\times \isom(\mathcal{T}_m)\to \mathcal{P}$. We claim $D^0$ is a multi-determinant functor.

  Indeed, fix an $i\in \{1,\dots,m\}$ and objects $a_k$ in $\mathcal{T}_k$ for $k \in \{1,\dots,m\}-\{i\}$. We will show $D^0$ is a determinant functor in the variable $i$. As before, we will notationally suppress the $a_k$ for convenience. 
  
  Suppose $x\to y\to z \to \Sigma x$ is a distinguished triangle in $\mathcal{T}_i$. Since it is by definition a $1$-cube 
  \begin{math}
    \begin{tikzcd}[cramped,sep=small]
      x \ar[r,dash] & y \ar[r,dash] & z 
    \end{tikzcd}
  \end{math}
  its image under $D$ is a $1$-cube in $\mathcal{P}$, i.e.
  \begin{equation*}
    D^0(z)+D^0(x)\xrightarrow{\cong} D^0(y).  
  \end{equation*}
  Moreover, the above isomorphism is natural with respect to isomorphisms of distinguished triangles because any isomorphic distinguished triangle fits into the $2$-cube
  \begin{equation*}
    \begin{tikzcd}
      x \arrow[r] & y \arrow[r] & z\\
      \ast \arrow[r] \arrow{u} & \ast \arrow[r] \arrow{u} & \ast \arrow{u}\\
      0 \arrow[r] \arrow{u} & 0 \arrow[r] \arrow{u} & 0 \arrow{u}
    \end{tikzcd}
  \end{equation*}
  and the image of this $2$-cube along with Lemma \ref{det_with_zero_as_an_input} gives the naturality condition. Note: the above diagram is a $2$-cube because it is an octahedron hence it admits a Verdier structure in an obvious way.  For the commutativity axiom consider:
  \begin{equation*}
    \begin{tikzcd}
      x \arrow{r} & x \arrow{r} & 0 \\
      x \arrow[r,"i"] \arrow{u} & x\oplus y \arrow[r,"p"] \arrow[u,"q"] & y \arrow{u} \\
      0 \arrow{r} \arrow{u} & y \arrow{r} \arrow[u,"j"] & y \arrow{u}
    \end{tikzcd}.
  \end{equation*}
  The above commutative diagram, whose rows and columns are distinguished triangles, is a $2$-cube because it admits the following Verdier structure:
  \begin{gather*}
    \begin{tikzcd}[ampersand replacement=\&]
      0 \arrow[r] \arrow[d, equal] \& x \arrow[r,"id"] \arrow[d,"i"] \& 
      x \arrow[d,"i"] \ar[r] \& 0 \ar[d,equal]\\
      0 \arrow[r] \& x\oplus y \arrow[r,"id"] \arrow[d,"p"] \& 
      x\oplus y \arrow[r] \arrow[d,"p"] \& 0\\
      \& y \arrow[r, equal] \arrow[d,"0"] \& y \arrow[d,"0"]\\
      \& \Sigma x \ar[r,"\id"] \& \Sigma x
    \end{tikzcd}
    \qquad
    \begin{tikzcd}[ampersand replacement=\&]
      0 \arrow[r] \arrow[d, equal] \& y \arrow[r,"id"] \arrow[d,"j"] \& 
      y \arrow[r] \arrow[d,"j"] \& 0 \arrow[d,equal] \\
      0 \arrow[r] \& x\oplus y \arrow[r,"id"] \arrow[d,"q"] \& 
      x\oplus y \arrow[r] \arrow[d,"q"] \& 0 \\
      \& x \arrow[r, equal] \ar[d,"0"] \& x \ar[d,"0"] \\
      \& \Sigma y \ar[r,"\id"] \& \Sigma y 
    \end{tikzcd}
    \\
    \begin{tikzcd}[ampersand replacement=\&]
      y \arrow[r,"j"] \arrow[d, equal] \& x\oplus y \arrow[r,"q"] \arrow[d,"p"] \& 
      x \arrow[r,"0"] \arrow[d] \& \Sigma y \ar[d,equal] \\
      y \arrow[r,"id"] \& y \arrow[r] \arrow[d,"0"] \& 
      0 \arrow[r] \arrow[d] \& \Sigma y \\
      \& \Sigma x \arrow[r, equal] \arrow[d,"id"] \& \Sigma x \arrow[d,"id"]\\
      \& \Sigma x \arrow[r,"id"] \& \Sigma x
    \end{tikzcd}
  \end{gather*}
  Thus the commutativity axiom is obtained using Lemma \ref{det_with_zero_as_an_input} and the image of aforementioned $2$-cube.  The octahedron axiom for determinant functors is obtained using Lemma \ref{det_with_zero_as_an_input} and the image of this $2$-cube representing the octahedron axiom for triangulated categories:
  \begin{equation*}
    \begin{tikzcd}
      0 \arrow{r} & x' \arrow{r} & x' \\
      x \arrow{r} \arrow{u} & z \arrow{r} \arrow{u} & y' \arrow{u} \\
      x \arrow{r} \arrow{u} & y \arrow{r} \arrow{u} & z' \arrow{u}
      \end{tikzcd}.
    \end{equation*}
  Next, we will show $D^0$ has compatibility between the variables: Suppose $x_i\to y_i\to z_i\to \Sigma x_i$ is a distinguished triangle in $\mathcal{T}_i$, $x_j\to y_j\to z_j\to \Sigma x_j$ is a distinguished triangle in $\mathcal{T}_j$, $f_i\colon u_i\to v_i$ is an isomorphism in $\mathcal{T}_i$ and $a_k$ are objects in in $\mathcal{T}_k$ for $k \in \{1,\dots,m\}-\{i, j\}$. We will continue to suppress the $a_k$ for notational convenience. The two triangle axiom follows from image of the $2$-cube
  \begin{equation*}
    \begin{tikzcd}
      (x_i,z_j) \arrow{r} & (y_i,z_j) \arrow{r} & (z_i,z_j) \\
      (x_i,y_j) \arrow{r} \arrow{u} & (y_i,y_j) \arrow{r} \arrow{u} & (z_i,y_j) \arrow{u} \\
      (x_i,x_j) \arrow{r} \arrow{u} & (y_i,x_j) \arrow{r} \arrow{u} & (z_i,x_j) \arrow{u}
    \end{tikzcd}
  \end{equation*}
  whereas the triangle-function axiom follows from Lemma \ref{det_with_zero_as_an_input} and the image of the $2$-cube
  \begin{equation*}
    \begin{tikzcd}
      (u_i,x_j) \arrow{r} & (u_i,y_j) \arrow{r} & (u_i,z_j) \\
      (v_i,x_j) \arrow{r} \arrow{u} & (v_i,y_j) \arrow{r} \arrow{u} & (v_i,z_j) \arrow{u} \\
      (0,x_j) \arrow{r} \arrow{u} & (0,y_j) \arrow{r} \arrow{u} & (0,z_j) \arrow{u}
    \end{tikzcd}
  \end{equation*}
  because $u_i\xrightarrow{f_i} v_i\to 0\to \Sigma u_i$ is a distinguished triangle.  In conclusion, $D^0$ is a multi-determinant functor.
\end{proof}

\begin{defn}
  \label{def_multifunctor_verdier}
  Suppose $F\colon\mathcal{T}_1\times\dots\times\mathcal{T}_n\to\mathcal{T}$ is a multiexact functor of triangulated categories.
  We say that $F$ \emph{admits a Verdier structure} if for all distinguished triangles $x_i\to y_i\to z_i\to \Sigma x_i$ in $\mathcal{T}_i$ and $x_j\to y_j\to z_j\to \Sigma x_j$ in $\mathcal{T}_j$, $i<j$, the induced $3\times 3$ diagram
  \begin{equation*}
    \begin{tikzcd}
      F( x_i, x_j) \arrow[r]\arrow[d] & F( y_i, x_j) \arrow[r]\arrow[d] & F( z_i, x_j) \arrow[r]\arrow[d] & \Sigma F( x_i, x_j)\arrow[d]\\
      F( x_i, y_j) \arrow[r]\arrow[d] & F( y_i, y_j) \arrow[r]\arrow[d] & F( z_i, y_j) \arrow[r]\arrow[d] & \Sigma F( x_i, y_j)\arrow[d]\\
      F( x_i, z_j) \arrow[r]\arrow[d] & F( y_i, z_j)\arrow[r]\arrow[d] & F( z_i, z_j)\arrow[r]\arrow[d]\arrow[dr,phantom,"\scalebox{0.75}{-1}"] & \Sigma F( x_i, z_j)\arrow[d]\\
      \Sigma F( x_i, x_j) \arrow[r] & \Sigma F( y_i, x_j) \arrow[r] & \Sigma F( z_i, x_j)\arrow[r] & \Sigma^2 F( x_i, x_j)
    \end{tikzcd}
  \end{equation*}
  admits a Verdier structure (in the sense of Definition \ref{Verdier_Structure}).
\end{defn}

\begin{cor}
  \label{exact_composed_with_det_is_a_multidet}
  Suppose $F\colon\mathcal{T}_1\times\dots\times\mathcal{T}_n\to\mathcal{T}$ is a multiexact functor of triangulated categories such that $F$ admits a Verdier structure. Then for every determinant functor $\GenericDet\colon\mathcal{T}\to\mathcal{P}$, the composition $\GenericDet\circ F$ is a multi-determinant functor.
\end{cor}

\begin{proof}
  Such a functor $F$ induces a functor $C^{\ast\leq 2}(\isom(\mathcal{T}_1),\dots,\isom(\mathcal{T}_n))\to C^{\ast\leq 2}(\isom(\mathcal{T}))$ that respects the face and degeneracy maps.  Then this result follows from Theorem \ref{defs_of_det_are_equivalent}.
\end{proof}

\subsection{Addition of Determinant Functors}

Addition of multi-determinant functors occurs pointwise, however, to make it more concrete, we will look at it through the lens of the cubical multi-determinant functors. First, we look at sums in $C^\ast(\mathcal{P})$.

\begin{defn}
  Let $\mathcal{P}$ be a Picard groupoid. The sum of two $n$-cubes $S,T\colon \objects(I^n)\to \mathcal{P}$ is an $n$-cube in $\mathcal{P}$ defined by
  \begin{equation*}
    (S+T)(a_1,\dots,a_n)\coloneqq S(a_1,\dots,a_n) + T(a_1,\dots,a_n).
  \end{equation*}
  Moreover, if the isomorphisms for $S$ are $f_1,\dots, f_n$ and the isomorphisms for $T$ are $g_1,\dots,g_n$, then the isomorphisms for $S+T$ are $(f_1+g_1)\circ\commutassoc,\dots,(f_n+g_n)\circ\commutassoc$.
\end{defn}

\begin{lemma}
  \label{sum_of_cubes_is_a_cube}
  The sum of two $n$-cubes (as defined above) is an $n$-cube.
\end{lemma}
\begin{proof}
  The only thing that needs to be shown is that a sum of $2$-cubes has the required additional structure but this follows from the additional structure on each $2$-cube and basic properties of the commutativity and associativity isomorphisms.
\end{proof}
In fact, more is true:
\begin{prop}
  For any $n$-cube $C \colon \objects(I^n)\to \mathcal{P}$ we have an isomorphism
  \begin{equation*}
    \partial^1_j C + \partial^{-1}_j C \xrightarrow{\cong} 
    \partial^0_j C\,,
  \end{equation*}
  for each direction $j=1,\dots,n$, where the isomorphism is determined by $\commutassoc$ or its inverse.
\end{prop}
\begin{proof}
  Straightforward.
\end{proof}

\begin{defn}
  For two cubical multi-determinant functors $\det_1,\det_2\colon C^{\leq 2}(\isom(\mathcal{T}_1)\times\dots\times\isom(\mathcal{T}_n)) \to C^{\leq 2}(\mathcal{P})$, the sum $\det_1+\det_2$ is defined by what it does on cubes $S\in C^{\leq 2}(\isom(\mathcal{T}_1)\times\dots\times\isom(\mathcal{T}_n))$:
  \begin{equation*}
    ({\det}_1+{\det}_2)(S) \coloneqq {\det}_1(S)+{\det}_2(S)\,.
  \end{equation*}
\end{defn}

\begin{thm}
  \label{sum_of_dets_is_a_det}
  The sum of two multi-determinant functors is itself a multi-determinant functor.
\end{thm}

\begin{proof}
The result immediately follows from Lemma \ref{sum_of_cubes_is_a_cube} and the fact that both cubical multi-determinant functors respect the face and degeneracy maps.
\end{proof}

\section{Universal Multi-determinant Functors}
\label{section_universal_multideterminants}

In this section, we prove the existence of universal multi-determinant functors, or, equivalently, that for any $n$-tuple of triangulated categories, the $2$-functor sending a Picard groupoid $\mathcal{P}$ to the groupoid of multi-determinants $D\colon \mathcal{T}_1\times \dots \times \mathcal{T}_n\to \mathcal{P}$ is representable. Recall that we use $\det\colon \mathcal{T}\to V (\mathcal{T})$ to denote a universal determinant functor; for $n$ triangulated categories, we keep the same notation for the juxtaposition of $n$ copies of those determinant functors:
\begin{equation*}
  \det \colon 
  \mathcal{T}_1\times\dots \times \mathcal{T}_n \to
  V(\mathcal{T}_1)\times\dots \times V(\mathcal{T}_n)\,.
\end{equation*}
We contend $\det$, as just (re)defined, is a universal multi-determinant. More precisely:

\begin{thm}
  \label{main_theorem_about_multi-determinants}
  For every multi-determinant functor $D\colon\mathcal{T}_0\times\mathcal{T}_1\times\dots\times\mathcal{T}_n\to \mathcal{P}$ from triangulated categories $\mathcal{T}_i$ to a Picard groupoid $\mathcal{P}$, there exists a pair $(d,a)$ comprised of a morphism in the multicategory of Picard groupoids $d\colon V(\mathcal{T}_0)\times V(\mathcal{T}_1)\times\dots\times V(\mathcal{T}_n)\to \mathcal{P}$ and a natural transformation $a\colon d\circ\det\Rightarrow D$.  This pair $(d,a)$ is unique in the following way: For any other pair $(g,\beta)$ with morphism $g\colon V(\mathcal{T}_0)\times\dots\times V(\mathcal{T}_n)\to \mathcal{P}$ and  natural transformation $\beta\colon g\circ\det\Rightarrow D$, there exists a unique natural transformation $\gamma\colon d\Rightarrow g$ such that $a = \beta\circ(\gamma\ast\det)$.
\end{thm}

Recall that if $\mathcal{M}$ is a multicategory we denote by $\mathcal{M}^\times$ the corresponding monoidal category of tuples of objects of $\mathcal{M}$. The Theorem \ref{main_theorem_about_multi-determinants} can be rephrased as follows:

\begin{cor}
  \label{categorical_equiv_of_DET_and_PIC}
  For any triangulated categories $\mathcal{T}_0,\mathcal{T}_1,\dots,\mathcal{T}_n$ and Picard groupoid $\mathcal{P}$, the pullback by the universal determinant $f\circ\det \leftmapsto f\colon\det^\ast$ gives an equivalence of groupoids
  \begin{equation*}
    \Det(\mathcal{T}_0,\dots,\mathcal{T}_n;\mathcal{P}) \simeq \Pic^\times(V(\mathcal{T}_0),\dots,V(\mathcal{T}_n);\mathcal{P})\,.
  \end{equation*}
\end{cor}
We denote the other direction of the equivalence map (from the above Corollary) by $\barfunc\colon\dfunc\mapsto\overline{\dfunc}$.

\subsection{Proof of Theorem \ref{main_theorem_about_multi-determinants}}

We will prove this Theorem using induction and dedicate the majority of this section to proving the pieces needed for the induction step. As such, the assumptions and notation below are used for the entirety of this section. \bigskip

%\noindent\textsc{Notation and Induction Assumptions}
\paragraph{Notation and induction assumptions}

\begin{itemize}
  \item (Function) Fix $n$ and $D\colon\mathcal{T}_0\times\mathcal{T}_1\times\dots\times\mathcal{T}_n\to \mathcal{P}$, a multi-determinant functor from triangulated categories $\mathcal{T}_i$ to Picard groupoid $\mathcal{P}$. This implies, for all $x\in\mathcal{T}_0$, there is a multi-determinant functor $D_x\colon\mathcal{T}_1\times\dots\times\mathcal{T}_n\to \mathcal{P}$ given by $D_x(t_1,\dots,t_n) = D(x,t_1,\dots,t_n)$.

  \item (Induced Morphisms) Any multi-determinant functor $d\colon\mathcal{T}_1\times\dots\times\mathcal{T}_n\to \mathcal{P}$ induces a pair $(\overline{d},\alpha)$ comprised of a morphism in the multicategory of Picard groupoids $\overline{d}\colon V(\mathcal{T}_1)\times\dots\times V(\mathcal{T}_n)\to \mathcal{P}$ and a natural transformation $\alpha\colon \overline{d}\circ\det\Rightarrow d$. This pair $(\overline{d},\alpha)$ is unique in the following way:

  \item (Uniqueness) For any other pair $(g,\beta)$ with morphism in the multicategory of Picard groupoids $g\colon V(\mathcal{T}_1)\times\dots\times V(\mathcal{T}_n)\to \mathcal{P}$ and  natural transformation $\beta\colon g\circ\det\Rightarrow d$, there exists a unique natural transformation $\gamma\colon \overline{d}\Rightarrow g$ such that $\alpha = \beta\circ(\gamma\ast\det)$. Pictorially,
  \begin{equation*}
    \begin{tikzcd}
      \mathcal{T}_1\times\dots\times\mathcal{T}_n \arrow{dd}[swap]{\det} \arrow[drr,"d"{name=D}]&& \\
      && \mathcal{P}\\
      V(\mathcal{T}_1)\times\dots\times V(\mathcal{T}_n) \arrow[urr, "\overline{d}", swap] \arrow[Rightarrow, shorten <= 10pt, shorten >= 10pt, to=D, "\alpha"] &&
    \end{tikzcd}
    \quad = \quad
    \begin{tikzcd}
      \mathcal{T}_1\times\dots\times\mathcal{T}_n \arrow{dd}[swap]{\det} \arrow[drr,"d"{name=D}]&& \\
      && \mathcal{P}\\
      V(\mathcal{T}_1)\times\dots\times V(\mathcal{T}_n) \arrow[urr, bend right=30, "\overline{d}"{name=B}, swap] \arrow[Rightarrow, shorten <= 10pt, shorten >= 10pt, to=D, "\beta"] \arrow[urr, "g"{name=G}] && \arrow[Rightarrow, shorten <= 5pt, shorten >= 5pt, from=B, to=G, "\gamma"]
    \end{tikzcd}
  \end{equation*}
\end{itemize}

%%%%%%%%%%%%%%%%%%%%%%%%%%%%%%%%%%%%%%%%%%%%%%%%%%%%%%%%%%%%%%%%%%%%%%%%%%%%%%%%%%%%%%%%%%%%%%%%%%%%%%%%%%%%%%%%%%%%%%%%%%%%%%%%%

%%%%%%%%%%%%%%%%%%%%%%%%%%%%%%%%%%%%%%%%%%%%%%%%%%%%%%%%%%%%%%%%%%%%%%%%%%%%%%%%%%%%%%%%%%%%%%%%%%%%%%%%%%%%%%%%%%%%%%%%%%%%%%%%%

\begin{lemma}
  \label{morphs_between_dets_induce_morphs}
  Let $\mathcal{T}_1,\dots,\mathcal{T}_n$ be triangulated categories, $\mathcal{P}$ be a Picard groupoid, $d_i\colon\mathcal{T}_1\times\dots\times\mathcal{T}_n\to\mathcal{P}$ be multi-determinant functors for $i=1,2$, and $(\overline{d_i},\alpha_i)$ be the pair induced by $d_i$. If there exists a morphism between determinant functors $\theta\colon d_1\Rightarrow d_2$, then there exists a unique monoidal natural transformation $\beta\colon \overline{d_2}\Rightarrow \overline{d_1}$ such that $\alpha_2 = \theta\circ \alpha_1\circ (\beta\ast\det)$.
\end{lemma}

\begin{proof}
This is a direct implication of the ``Induced Morphism'' and ``Uniqueness Assumption''.
\end{proof}

Note: although the above lemma follows directly from our assumptions, we include it as a reminder of the arrow reversal for the function $\barfunc$.

\begin{prop}
  \label{induced_morphisms_respect_addition}
  For any $x,y\in\mathcal{T}_0$ we have $\overline{D_x+D_y} \xrightarrow{\cong} \overline{D_x}+\overline{D_y}$.
\end{prop}

\begin{proof}
  Let $(\overline{D_x},\alpha_x)$ and $(\overline{D_y},\alpha_y)$ be the induced pairs of $D_x$ and $D_y$, respectively.  By Theorem \ref{sum_of_dets_is_a_det}, $D_x+D_y$ is a determinant functor, which means it has an induced pair $(\overline{D_x+D_y},\beta)$. 

  Consider the pair $(\overline{D_x}+\overline{D_y},\alpha_x+\alpha_y)$ where all sums are defined by pointwise addition. Because $\mathcal{P}$ is a Picard groupoid and addition is defined pointwise,  $\overline{D_x}+\overline{D_y}\in\Pic(V(\mathcal{T}_1),\dots,V(\mathcal{T}_n);\mathcal{P})$. Notice $\alpha_x+\alpha_y$ is a natural transformation $(\overline{D_x}+\overline{D_y})\circ\det\Rightarrow D_x+D_y$; indeed, for any object $k$ in $\mathcal{T}_1\times\dots\times\mathcal{T}_n$ we have:
  \begin{align*}
    ((\overline{D_x}+\overline{D_y})\circ\det) (k) &= (\overline{D_x}+\overline{D_y})(\det(k))\\
    &\overset{\text{def}}{=} \overline{D_x}(\det(k)) + \overline{D_y}(\det(k))\\
    &= (\overline{D_x}\circ\det)(k) + (\overline{D_y}\circ\det)(k)\\
    &\xrightarrow{\alpha_x+\alpha_y} D_x(k)+D_y(k)\\
    &\overset{\text{def}}{=} (D_x+D_y)(k)
  \end{align*}
  and for any morphism $f\colon k\to m$ in $\mathcal{T}_1\times\dots\times\mathcal{T}_n$, the diagram
  \begin{equation*}
    \begin{adjustbox}{scale=0.85}
    \begin{tikzcd}[column sep=large]
      ((\overline{D_x}+\overline{D_y})\circ\det)(k) = 
      (\overline{D_x}\circ\det)(k) + (\overline{D_y}\circ\det)(k)
      \arrow[d, "\alpha_x(f)+\alpha_y(f)", swap] 
      \arrow[rr, "\alpha_x(k)+\alpha_y(k)"] 
      & & D_x(k)+D_y(k) = (D_x+D_y)(k) 
      \arrow[d, "\alpha_x(f)+\alpha_y(f)"]  \\
      ((\overline{D_x}+\overline{D_y})\circ\det)(m) = (\overline{D_x}\circ\det)(m) + (\overline{D_y} \circ\det)(m) 
      \arrow[rr, "\alpha_x(m)+\alpha_y(m)", swap] 
      & & D_x(m)+D_y(m) = (D_x+D_y)(m)
    \end{tikzcd}
    \end{adjustbox}
  \end{equation*}
  is commutative since $\alpha_x$ and $\alpha_y$ are natural transformations.

  Therefore, by the Uniqueness Assumption, there exists a unique natural transformation 
  \begin{equation*}
    \gamma\colon \overline{D_x+D_y}\Rightarrow \overline{D_x}+\overline{D_y}
  \end{equation*} 
  such that $\beta = (\alpha_x+\alpha_y)\circ(\gamma\ast\det)$. Since $\gamma$ is a natural transformation in $\Pic$, then it is an isomorphism.
\end{proof}

\begin{prop}
  \label{arrow_in_DET_induces_arrow_in_Pic(V,P)}
  Every morphism $f\colon d_1\to d_2$ in $\Det(\mathcal{T}_1,\dots,\mathcal{T}_n;\mathcal{P})$ induces a unique (in the sense of the Uniqueness Assumption) arrow-reversing monoidal morphism $\overline{f}\colon \overline{d_2}\to \overline{d_1}$ in the Picard groupoid $\Pic(V(\mathcal{T}_1),\dots,V(\mathcal{T}_n);\mathcal{P})$. Moreover, if $(\overline{d_1},\alpha_1)$ and $(\overline{d_2},\alpha_2)$ be the induced pairs of $d_1$ and $d_2$ respectively, then $\alpha_2 = f\circ\alpha_1\circ (\overline{f}\ast\det)$.
\end{prop}

\begin{proof}
  This follows from the Uniqueness Assumption. Indeed, use the natural transformations $\alpha_2\colon \overline{d_2}\circ\det\Rightarrow d_2$ and $f\circ\alpha_1\colon \overline{d_1}\circ\det\Rightarrow d_2$. Then by the Uniqueness Assumption on $(\overline{d_2},\alpha_2)$, there exists a unique natural transformation $\overline{f}\colon \overline{d_2}\Rightarrow \overline{d_1}$ such that $\alpha_2 = f\circ\alpha_1\circ (\overline{f}\ast\det)$.
\end{proof}

\begin{prop}
  \label{composition_in_DET_induces_composition_in_Pic(V,P)}
  Every composable pair of morphisms $g\circ f\colon x\to y\to z$ in $\Det(\mathcal{T}_1,\dots,\mathcal{T}_n;\mathcal{P})$ induces composition in $\Pic(V(\mathcal{T}_1),\dots,V(\mathcal{T}_n);\mathcal{P})$, in other words $\overline{g\circ f} = \overline{f} \circ \overline{g}$.
\end{prop}

\begin{proof}
  Let $(\overline{x},\alpha_x)$, $(\overline{y},\alpha_y)$ and $(\overline{z},\alpha_z)$ be the induced pairs of $x$, $y$, and $z$ respectively. Using the morphism $g\circ f\colon x\to z$, Proposition \ref{arrow_in_DET_induces_arrow_in_Pic(V,P)} tells us there exists a unique morphism $\overline{g\circ f}\colon \overline{z}\to \overline{x}$ such that

  \begin{equation*}
    \alpha_z = (g\circ f)\circ\alpha_x\circ (\overline{g\circ f}\ast\det)\,.
  \end{equation*}
  However, there is a ``second'' arrow $\overline{z}\to \overline{x}$ and ``second'' way to rewrite $\alpha_z$. Indeed, by Proposition \ref{arrow_in_DET_induces_arrow_in_Pic(V,P)}, there exists unique morphisms $\overline{f}\colon \overline{y}\to \overline{x}$ and $\overline{g}\colon \overline{z}\to \overline{y}$, such that we can rewrite $\alpha_y = f\circ\alpha_x\circ (\overline{f}\ast\det)$ and $\alpha_z = g\circ\alpha_y\circ (\overline{g}\ast\det)$; thus, by plugging in the equation for $\alpha_y$, we get a ``second'' way to rewrite $\alpha_z$:
  \begin{align*}
    \alpha_z &= g\circ(f\circ\alpha_x\circ (\overline{f}\ast\det))\circ (\overline{g}\ast\det)\\
    &= (g\circ f)\circ\alpha_x\circ ((\overline{f}\circ\overline{g})\ast\det).
  \end{align*}
  for a ``second'' morphism $(\overline{f}\circ\overline{g})\colon \overline{z}\to\overline{x}$.  Now the uniqueness of $\overline{g\circ f}$ forces the equality $\overline{g\circ f} = \overline{f} \circ \overline{g}$.
\end{proof}

\begin{cor}
  \label{diagram_in_DET_induces_diagram_in_Pic(V,P)}
  Any commutative diagram in $\Det(\mathcal{T}_1,\dots,\mathcal{T}_n;\mathcal{P})$ induces a commutative diagram (with reversed arrows) in $\Pic(V(\mathcal{T}_1),\dots,V(\mathcal{T}_n);\mathcal{P})$ using the map $f\mapsto \overline{f}$.
\end{cor}

\begin{proof}
  This is a direct result of Propositions \ref{arrow_in_DET_induces_arrow_in_Pic(V,P)} and \ref{composition_in_DET_induces_composition_in_Pic(V,P)}.
\end{proof}

\begin{prop}
  \label{x_mapsto_D_x_is_a_determinant_functor}
  The arrow $\mathcal{T}_0\to\Det(\mathcal{T}_1,\dots,\mathcal{T}_n;\mathcal{P})$ given by $x\mapsto D_x$ is a determinant functor.
\end{prop}

\begin{proof}
  Since $D$ is a multi-morphism compatibly natural in all variables (specifically, it has the properties of a bifunctor), then for any $f\colon x\to y$ in the first variable, the arrows 
  \begin{equation*}
    D_f(z_1,\dots,z_n)\colon D(x,z_1,\dots,z_n)\to D(y,z_1,\dots,z_n)
  \end{equation*}
  define a natural transformation $D_f\colon D_x\Rightarrow D_y$. Moreover, $D_f$ is compatible with the additivity data by the Triangle-Function Axiom of multi-determinant functors. Therefore $f$ induces a morphism of multi-determinant functors $D_f\colon D_x\Rightarrow D_y$.

  For any distinguished triangle $\Delta\colon x\to y\to z\to \Sigma x$ in $\mathcal{T}_0$, there is a natural transformation $\beta\colon D_z+D_x\Rightarrow D_y$ since $D$ is a determinant functor in its first variable and by the Triangle-Function Axiom. Moreover, $\beta$ is compatible with the additivity data by the Two-Triangle Axiom of multi-determinant functors. Therefore the triangle $\Delta$ induces a morphism of multi-determinant functors $D_z+D_x\Rightarrow D_y$.

  To see that $x\mapsto D_x$ induces the Commutativity and Octahedron Axioms, notice for any fixed $(x_1,\dots,x_n)\in\mathcal{T}_1\times\dots\times\mathcal{T}_n$ all of the requirements for determinant functors are satisfied because $D$ is a multi-determinant functor and hence a determinant functor in each variable. For any morphisms or triangles in $\isom(\mathcal{T}_1)\times\dots\times\isom(\mathcal{T}_n)$, the Two Triangles Axiom and Triangle-Function Axiom of multi-determinant functors give the necessary compatibility.
\end{proof}

\begin{cor}
  \label{x_mapsto_overline{D_x}_is_a_determinant_functor}
  The arrow $\mathcal{T}_0\to\Pic(V(\mathcal{T}_1),\dots,V(\mathcal{T}_n);\mathcal{P})$ given by $x\mapsto \overline{D_x}$ is a determinant functor.
\end{cor}

\begin{proof}
  The arrow $x\mapsto \overline{D_x}$ is a composition of two arrows. The first arrow $x\mapsto D_x$, is a determinant functor by Proposition \ref{x_mapsto_D_x_is_a_determinant_functor} and is followed by the arrow $f\mapsto \overline{f}$, which is a (contravariant) functor by Lemma \ref{morphs_between_dets_induce_morphs} and Proposition \ref{composition_in_DET_induces_composition_in_Pic(V,P)}. By Propositions \ref{induced_morphisms_respect_addition} and \ref{arrow_in_DET_induces_arrow_in_Pic(V,P)}, for any distinguished triangle $x\to y\to z\to \Sigma x$ in $\mathcal{T}_0$, $\overline{D_z}+\overline{D_x}\overset{\cong}{\leftarrow} \overline{D_z+D_x}\leftarrow \overline{D_y}$. Moreover, since $x\mapsto D_x$ satisfies the Commutativity and Octahdron Axioms, then by Corollary \ref{diagram_in_DET_induces_diagram_in_Pic(V,P)} $x\mapsto\overline{D_x}$ does too.
\end{proof}

\begin{proof}[End of proof of Theorem \ref{main_theorem_about_multi-determinants}]
  When $n=0$, we are looking at a determinant functor $\mathcal{T}_0\to\mathcal{P}$. It is a known fact that such a determinant functor induces the desired unique pair (for example, see \cite{muro2010determinant}).

  Now assume the Theorem holds for $n$ triangulated categories, i.e. assume the induction assumptions from the beginning of the section hold. By Corollary \ref{x_mapsto_overline{D_x}_is_a_determinant_functor}, the arrow
  \begin{equation*}
    \Dbar\colon \mathcal{T}_0\to \Pic(V(\mathcal{T}_1),\dots,V(\mathcal{T}_n);\mathcal{P})
  \end{equation*}
  given by $\Dbar\colon x\mapsto \overline{D_x}$ is a determinant functor. Therefore, there exists an induced functor $\overline{\Dbar}\colon V(\mathcal{T}_0)\to\Pic(V(\mathcal{T}_1),\dots, V(\mathcal{T}_n);\mathcal{P})$ of Picard groupoids along with a natural transformation $\overline{\alpha}\colon \overline{\Dbar}\circ\det\Rightarrow \Dbar$ that is unique (in the sense of the Uniqueness Assumption). 

  Since the multicategory of Picard groupoids is closed, then $(\overline{\Dbar},\overline{\alpha})$ induces maps $(d,a)$ where $d\in \Pic(V(\mathcal{T}_0),\dots,V(\mathcal{T}_n);\mathcal{P})$ and $a\colon d\circ\det\Rightarrow D$, which also has the uniqueness assumption.
\end{proof}

\subsection{Functoriality of the Equivalence} %Inducing Structure from TRCAT}
\label{sec:funct_equiv}

Recall that $\TrCat$ denotes the multicategory (in the enriched sense) of (small) triangulated categories. Also recall Definition \ref{def_multifunctor_verdier}, concerning multifunctors with Verdier structure. We slightly restrict $\TrCat$ to only consider those multi morphisms:
\begin{defn}
  \label{def_trcat_verdier}
  The multicategory $\TrCat_{\mathcal{V}}$ is the sub-multicategory of $\TrCat$ in which all functors admit a Verdier structure. Analogously for the induced monoidal category $\TrCat^\times_{\mathcal{V}}$.
\end{defn}
  
\begin{thm}
  \label{Det_is_functorial_in_T}
  For any Picard groupoid $\mathcal{P}$ and all $n$-tuples $\underline{\mathcal{T}} =(\mathcal{T}_1,\dots, \mathcal{T}_n)$, for all $n\geq 0$, the groupoids $\Det(\underline{\mathcal{T}},\mathcal{P})$ determine a 2-functor
  \begin{equation*}
    \Det(-;\mathcal{P})\colon (\TrCat^{\times}_{\mathcal{V}})^{\opposite}\to \Pic\,.
  \end{equation*}
\end{thm}

\begin{proof}
  The codomain of $\Det$ is automatically the category of groupoids. Since there is an equivalence of 2-categories $\Det(\mathcal{T}_1,\dots,\mathcal{T}_n;\mathcal{P}) \simeq \Pic^\times(V(\mathcal{T}_1),\dots,V(\mathcal{T}_n);\mathcal{P})$ and the category of Picard groupoids is closed, then the codomain is actually the category of Picard groupoids.

  Using the same reasoning as Corollary \ref{exact_composed_with_det_is_a_multidet}, for any triangulated categories $\mathcal{T}_1,\dots,\mathcal{T}_n$ and any multiexact morphisms
  $F_i\colon \mathcal{S}^i_1\times\dots\times \mathcal{S}^i_{m_i}\to \mathcal{T}_i$ that admit a Verdier structure, the morphism $(F_1,\dots,F_n)$ induces a map
  \begin{equation*}
    (F_1,\dots,F_n)^\ast\colon \Det(\mathcal{T}_1,\dots,\mathcal{T}_n;\mathcal{P})\to \Det(\mathcal{S}^1_1,\dots,\mathcal{S}^n_{m_n};\mathcal{P}),\, D\mapsto D\circ(F_1,\dots,F_n)\,.
  \end{equation*}
  Any morphism of multi-determinant functors $\eta\colon D\to E$ gives a morphism of determinant functors $\eta\ast(F_1,\dots,F_n)\colon D\circ(F_1,\dots,F_n)\to E\circ (F_1,\dots,F_n)$. Moreover, the identity and composition is clearly preserved.
\end{proof}

\begin{thm}
  \label{equiv_of_categories_is_functorial_in_T}
  For any Picard groupoid $\mathcal{P}$, the equivalence
  \begin{equation*}
    \Det(\mathcal{T}_1,\dots,\mathcal{T}_n;\mathcal{P}) \simeq \Pic^\times(V(\mathcal{T}_1),\dots,V(\mathcal{T}_n);\mathcal{P})
  \end{equation*}
  is functorial in $\mathcal{T}_1,\dots,\mathcal{T}_n\in \TrCat^{\times}_{\mathcal{V}}$.
\end{thm}
\begin{note}
  For any determinant functor $d$, we will use $\overline{d}$ to denote the induced multiexact functor between Picard groupoids guaranteed by Theorem \ref{main_theorem_about_multi-determinants}.
\end{note}
\begin{proof}
  Let $\mathcal{T}_1,\dots,\mathcal{T}_n$ be triangulated categories and  $F_i\colon \mathcal{S}^i_1\times\dots\times \mathcal{S}^i_{m_i}\to \mathcal{T}_i$ be multiexact functors that admit a Verdier structure. For a determinant functor $D \colon \mathcal{T}_1\times\dots\times\mathcal{T}_n \to \mathcal{P}$ and morphism $f \colon (V(\mathcal{T}_1),\dots,V(\mathcal{T}_n)) \to \mathcal{P}$, we will use $(F_1,\dots,F_n)^\ast$ to denote the map $(F_1,\dots,F_n)^\ast(D) = D\circ (F_1,\dots,F_n)$, and $\overline{(F_1,\dots,F_n)^\ast}$ to denote the map $\overline{(F_1,\dots,F_n)^\ast}(f) = f\circ(\overline{\det\circ F_1},\dots,\overline{\det\circ F_n})$. 

  To see that the equivalence preserves morphisms, i.e.\ that the diagram
  \begin{equation*}
    \begin{tikzcd}
      \Det(\mathcal{T}_1,\dots,\mathcal{T}_n;\mathcal{P}) \arrow[r, "\simeq"] \arrow[d, "{(F_1,\dots,F_n)^\ast}", swap]
      & \Pic^\times(V(\mathcal{T}_1),\dots,V(\mathcal{T}_n);\mathcal{P}) \arrow[d, "\overline{(F_1,\dots,F_n)^\ast}"]\\
      \Det(\mathcal{S}^1_1,\dots,\mathcal{S}^n_{m_n};\mathcal{P}) \arrow[r, "\simeq"]
      & \Pic^\times(V(\mathcal{S}^1_1),\dots,V(\mathcal{S}^n_{m_n});\mathcal{P})
    \end{tikzcd}
  \end{equation*}
  commutes up to a natural transformation, we will show that for every multi-determinant functor $D\colon \mathcal{T}_1\times\dots\times \mathcal{T}_n\to \mathcal{P}$ there exists a natural transformation $\beta\colon \overline{D}\circ(\overline{\det\circ F_1},\dots,\overline{\det\circ F_n})\circ\det\Rightarrow D\circ(F_1,\dots,F_n)$. This natural transformation along with the uniqueness of Theorem \ref{main_theorem_about_multi-determinants} will induce the desired natural transformation $\overline{D\circ(F_1,\dots,F_n)} \Rightarrow \overline{D}\circ(\overline{\det\circ F_1},\dots,\overline{\det\circ F_n})$.

  Let $\alpha_i$ be the natural transformation $\overline{\det\circ F_i}\circ\det\Rightarrow \det\circ F_i$, and $\tau$ the natural transformation $\overline{D}\circ\det\to D$, both of which are guaranteed to exist by Theorem \ref{main_theorem_about_multi-determinants}. Then we can take $\beta$ to be $$(\tau\ast (F_1,\dots,F_n))\circ(\overline{D}\ast(\alpha_1,\dots,\alpha_n)).$$

  Moreover, for any multiexact functors $G_{i,j}\colon \mathcal{R}^{i,j}_1\times\dots\times\mathcal{R}^{i,j}_{p_j}\to \mathcal{S}^i_j$ that admits a Verdier structure with natural transformations $\gamma_{i,j}\colon \overline{\det\circ G_{i,j}}\circ\det\Rightarrow \det\circ G_{i,j}$, the composition $(F_1,\dots,F_n)\circ(G_{1,1},\dots,G_{n,m_n})$ is preserved. Indeed, use the natural transformation
  \begin{align*}
    (\tau\ast((F_1,\dots,F_n)\circ(G_{1,1},\dots,G_{n,m_n})))
    &\circ (\overline{D}\ast(\alpha_1,\dots,\alpha_n)\ast(G_{1,1},\dots,G_{n,m_n}))\\
    &\circ (\overline{D}\ast(\overline{\det\circ F_1},\dots,\overline{\det\circ F_n})\ast (\gamma_{1,1},\dots\gamma_{n,m_n}))
  \end{align*}
  to induce the desired natural transformation. Lastly, we note that the identity is clearly preserved.
\end{proof}

\begin{prop}
  \label{2-morphism_in_TRCAT_induce_morphisms_in_DET}
  For any triangulated categories $\mathcal{T}_0,\mathcal{T}_1,\dots,\mathcal{T}_n$ and Picard groupoid $\mathcal{P}$, if $D\colon \mathcal{T}_0\to\mathcal{P}$ is a determinant functor, then any natural isomorphism $\alpha\colon F_1\to F_2$ between multiexact functors that admit Verdier structures $F_i\colon \mathcal{T}_1\times\dots\times\mathcal{T}_n\to \mathcal{T}_0$ induces a morphism of multi-determinant functors $D\ast\alpha\colon D\circ F_1\to D\circ F_2$.
\end{prop}

\begin{proof}
  By Corollary \ref{exact_composed_with_det_is_a_multidet}, $D\circ F_i$ is a multi-determinant functor. Moreover, $D\ast\alpha$ is automatically a natural transformation. Lastly, since every distinguished triangle $\Delta_i$ in $\mathcal{T}_i$ yields a natural isomorphism $\alpha(\Delta_i)\colon F_1(\Delta_i)\to F_2(\Delta_i)$ between distinguished triangles and $D$ is a determinant functor, then $D\ast\alpha$ is compatible with the additivity data.
\end{proof}

\begin{thm}
  \label{induce_structure_from_TRCAT}
  Let $\mathcal{T}_0,\mathcal{T}_1,\dots,\mathcal{T}_n$ be triangulated categories and $\mathcal{P}$ be a Picard groupoid. Suppose $D\colon \mathcal{T}_0\to\mathcal{P}$ is a determinant functor and $\alpha\colon F_1\to F_2$ is a natural transformation between multiexact functors that admit Verdier structures $F_i\colon \mathcal{T}_1\times\dots\times\mathcal{T}_n\to \mathcal{T}_0$. Then there is an induced natural transformation $\overline{D\ast\alpha}\colon \overline{D\circ F_2}\to \overline{D\circ F_1}$ between multiexact morphisms $\overline{D\circ F_i}\colon V(\mathcal{T}_1)\times\dots\times V(\mathcal{T}_n)\to\mathcal{P}$.
\end{thm}

\begin{proof}
  This follows immediately from Proposition \ref{2-morphism_in_TRCAT_induce_morphisms_in_DET} and Corollary \ref{categorical_equiv_of_DET_and_PIC}.
\end{proof}

\section{Determinant Functors on Tensor Triangulated Categories}
\label{sec:det_ttt}

\begin{defn}
  \label{tt-definition}
  A tensor triangulated category is a triple $(\mathcal{T},\otimes,\mathbf{1})$ consisting of a triangulated category $\mathcal{T}$, a monoidal structure $\otimes \colon \mathcal{T}\times \mathcal{T} \to \mathcal{T}$, with unit object $\mathbf{1}$, such that tensor is a biexact functor in the sense specified in Definition \ref{def:multiexact_for_triangulated_categories}.
\end{defn}

\begin{rmk}
  \label{commutativity-not-needed}
  We have not included the symmetry condition, which is usually assumed (see, e.g.\ \cite{balmer2010ttg}), in the definition. We do not emphasize the commutativity of the tensor operation below, but note that Definition~\ref{tt-definition} coincides with Balmer's if the monoidal structure is indeed symmetric.
\end{rmk}

\begin{rmk}
  A tensor triangulated category is like a \emph{2-rig,} that is, a category with two monoidal structures satisfying the categorified version of the axioms of a rig, namely a ring without negative elements. This is because the multiplicative structure (given by the tensor $\otimes$) distributes over the additive one (given by the coproduct $\oplus$) due to the bi-exactness of $\otimes$ and the properties of the triangulation. For instance, for any objects $x,y$ and $z$, consider the diagram
  \begin{equation*}
    \begin{tikzcd}
      x\otimes z \arrow[r] \arrow[d,equal] & 
      (x\otimes z)\oplus (y\otimes z) \arrow[r] \arrow[d,dotted,"\exists","\cong"'] &
      y\otimes z \arrow[r] \arrow[d,equal] & 
      \Sigma x\otimes z \arrow[d,"\cong"]\\
      x\otimes z \arrow[r] & (x\oplus y)\otimes z \arrow[r] & y \otimes z \arrow[r] & \Sigma x\otimes z
    \end{tikzcd}
  \end{equation*}
  Where the bottom triangle is distinguished because so is $x \to x \oplus y \to y \to \Sigma x$ and $\otimes$ is biexact. Then the existence of the ``distributor'' isomorphism (the second vertical arrow) follows from axioms TR3 and TR4. The other structural morphisms of a 2-rig follow from similar considerations, but there is not a canonical choice for these morphisms. 
\end{rmk}

\begin{defn}
  The tensor $\otimes$ in a tensor triangulated category $\mathcal{T}$ \emph{admits a Verdier structure} if for all distinguished triangles $x\overset{f}{\to} y\overset{g}{\to} z\overset{h}{\to} \Sigma x$ and $u\overset{p}{\to} v\overset{q}{\to} w\overset{r}{\to} \Sigma u$ in $\mathcal{T}$, the $3\times 3$ diagram
  \begin{equation*}
    \begin{tikzcd}
      x\otimes u\arrow[r,"f\otimes id"]\arrow[d,"id \otimes p"] & 
      y\otimes u \arrow[r,"g\otimes id"]\arrow[d,"id \otimes p"] & 
      z\otimes u \arrow[r,"h\otimes id"]\arrow[d,"id \otimes p"] & 
      \Sigma x\otimes x\arrow[d,"id \otimes \Sigma f"]\\
      x \otimes v\arrow[r,"f\otimes id"]\arrow[d,"id \otimes q"] & 
      y\otimes v\arrow[r,"g \otimes id"]\arrow[d,"id \otimes q"] & 
      z\otimes v\arrow[r,"h\otimes id"]\arrow[d,"id \otimes q"] & 
      \Sigma x\otimes v\arrow[d,"id \otimes\Sigma q"]\\
      x \otimes w\arrow[r,"f \otimes id"]\arrow[d,"id \otimes r"] & 
      y \otimes w\arrow[r,"g\otimes id"]\arrow[d,"id \otimes r"] & 
      z\otimes w\arrow[r,"h \otimes id"]\arrow[d,"id \otimes r"]\arrow[dr,phantom,"\scalebox{0.75}{-1}"] & 
      \Sigma x\otimes w\arrow[d,"id \otimes \Sigma r"]\\
      \Sigma x\otimes u\arrow[r,"\Sigma f\otimes id"] & 
      \Sigma y\otimes u\arrow[r,"\Sigma g\otimes id"] & 
      \Sigma z\otimes u\arrow[r,"\Sigma h\otimes id"] & \Sigma^2 x\otimes u
    \end{tikzcd}
  \end{equation*}
  admits a Verdier structure (in the sense of Definition \ref{Verdier_Structure}).
\end{defn}

\begin{rmk}
The axioms considered by Keller and Neeman in \cite[Definition 3.1]{keller2002connection} imply that their tensor product admits a Verdier structure. (They call such a tensor triangulated category ``decent.'') These axioms were previously also considered by P.\ May \cite{may2001ttt}. These extra axioms are not included in Balmer's definition of tensor triangulated category \cite{balmer2010ttg,balmer2010sss,balmer2005prime}, nor in Definition~\ref{tt-definition}.
From now on, we assume the tensor structure satisfies this additional property; however, in order to not cause confusion, we do not modify the definition, but rather add this explicit requirement to the statements when needed. (Cf.\ Theorem~\ref{cat_ring_for_tensor_triangulated} below.)
\end{rmk}

\begin{defn}
  We call a Picard groupoid $(\mathcal{P},\picadd,0)$ a \emph{categorical ring} when there exists a second monoidal structure in the form of a multiexact unital associative map $\picmult\colon\mathcal{P}\times\mathcal{P}\to\mathcal{P}$.
\end{defn}

\begin{rmk}
  The multiexactness of the second monoidal structure refers to the biexactness of the operation $\picmult \colon \mathcal{P}\times \mathcal{P} \to \mathcal{P}$ as well as to the composite operations $\picmult \circ (\picmult \times \id_\mathcal{P})$, $\picmult \circ (\id_\mathcal{P} \times \picmult)$, etc.\ which intervene in the axioms. These can be expressed in ``unbiased form'' by introducing appropriate higher operations $\picmult_n\colon \mathcal{P}^{\times_n} \to \mathcal{P}$ which will be assumed to be multiexact relative to the underlying Picard groupoid structure $(\mathcal{P},\picadd,0)$ (see \cite{aldrovandi2015biex}).
\end{rmk}

\begin{rmk}
  \label{pi_0_and_pi_1_of_categorical_rings}
  It is well known (see e.g.\ \cite{MR2369166}) that if $\mathcal{P}$ is a categorical ring, then $\pi_0(\mathcal{P})$ is a ring and $\pi_1(\mathcal{P})$ is a $\pi_0$-bimodule. 
  
  To be more specific, if $(\mathcal{P},\picadd, 0, \picmult, 1)$ is a categorical ring, then $\pi_0(\mathcal{P})$ is the abelian group $\objects(\mathcal{P})/\sim$ where $x\sim y$ when $x$ and $y$ are isomorphic, and the operations are $[x]+[y]\coloneqq [x\picadd y]$, and $[x]\cdot[y]\coloneqq [x\picmult y]$. 
  
  Additionally, $\pi_1(\mathcal{P})$ is $\autom_{\mathcal{P}}(0)$. This is an abelian group under composition and has the following actions: For $f\in \autom(0)$, define the left action $A\action f$ by
  \begin{equation*}
    \begin{tikzcd}
      A\picmult 0 \arrow[r, "_{A}\sigma"] \arrow[d, "\identity_A\picmult f", swap] & 0 \arrow[d, dashed, "A\action f"] \\
      A\picmult 0 \arrow[r, "_{A}\sigma", swap] & 0 \arrow[ul, phantom, "\circ"]
    \end{tikzcd}
  \end{equation*}
  where ${}_A\sigma\colon A\picmult 0\to 0$ is the composition
  \begin{equation*}
    \begin{split}
      A\picmult 0 \to (A\picmult 0)\picadd 0 &\to (A\picmult 0)\picadd [(A\picmult 0)\picadd(-(A\picmult 0))]\to [(A\picmult 0)\picadd (A\picmult 0)] \picadd (-(A\picmult 0))\to \\
      &\to [A\picmult (0\picadd 0)]\picadd (-(A\picmult 0)) \to (A\picmult 0)\picadd (-(A\picmult 0)) \to 0.
    \end{split}
  \end{equation*}
  Similarly let $\sigma_A\colon 0\picmult A \to 0$ be the analogous composition; then the right action $f\action A$ is similarly defined using $\sigma_A$.
\end{rmk}

\begin{thm}
  \label{cat_ring_for_tensor_triangulated}
  If $\mathcal{T}$ is a tensor triangulated category whose tensor admits a Verdier structure, then the universal Picard groupoid $V(\mathcal{T})$ of virtual objects is a categorical ring.
\end{thm}

\begin{proof}
  Let $\mathcal{T}$ be a tensor triangulated category where $\otimes\colon\mathcal{T}\times\mathcal{T}\to\mathcal{T}$ is the tensor structure that admits a Verdier structure. In particular, $\otimes$ is unital associative and biexact.
  %and compatible with the monoidal structure from the triangulated category.
  Let $\picmult\colon V(\mathcal{T})\times V(\mathcal{T})\to V(\mathcal{T})$ be induced from $\otimes$ (as in Theorem \ref{main_theorem_about_multi-determinants}). Since $\otimes$ is multiexact and admits a Verdier structure, then $\otimes\circ \det$ is a multi-determinant functor by Corollary \ref{exact_composed_with_det_is_a_multidet}. Hence $\picmult$ is automatically multiexact.

  Notice that all unital and associative properties from $\otimes$ can be phrased in the form
  \begin{equation*}
    \begin{tikzcd}
      \mathcal{T}\times\dots\times\mathcal{T} \arrow[rr, bend left=40, ""{name=A, below}] \arrow[rr, bend right=40, ""{name=B}] & & \mathcal{T} \arrow[Rightarrow, from=A, to=B, "\cong"]
    \end{tikzcd}
  \end{equation*}
  for some functors that admit Verdier structures and a natural isomorphism. These natural transformations transfer to natural transformations between multiexact functors in the multicategory of Picard groupoids (Proposition \ref{2-morphism_in_TRCAT_induce_morphisms_in_DET} and Theorem \ref{induce_structure_from_TRCAT}). Moreover, their images are the corresponding unital and associative identities for $\picmult$ because $\Det$ and the equivalence of categories between $\Det$ and $\Pic$ are functorial in $\mathcal{T}$ (Theorems \ref{Det_is_functorial_in_T} and \ref{equiv_of_categories_is_functorial_in_T}).
\end{proof}

\begin{rmk}
  If $\mathbf{1}$ is the unit of a tensor triangulated category $\mathcal{T}$, then $\det(\mathbf{1})$ is the unit of the categorical ring $V(\mathcal{T})$.
\end{rmk}

\begin{cor}
  If $\mathcal{T}$ is a tensor triangulated category whose tensor admits a Verdier structure, then $K_0(\mathcal{T})$ is a ring and $K_1(\mathcal{T}) \coloneqq \pi_1(V(\mathcal{T}))$ is a $K_0$ bi-module 
\end{cor}

\begin{proof}
  This is an immediate consequence of Theorems \ref{cat_ring_for_tensor_triangulated}, Remark \ref{pi_0_and_pi_1_of_categorical_rings}, and the fact that $K_0(\mathcal{T})$ is $\pi_0(V(\mathcal{T}))$.
\end{proof}

\bibliographystyle{hamsalpha}
\bibliography{references}
\end{document}